\documentclass[a4paper, 11pt, reqno]{amsart} 
\usepackage{mathrsfs}
\usepackage[normalem]{ulem}
\usepackage{amsmath, amsthm, thmtools, amssymb, tikz} 
\usepackage{bbm}
\usepackage{bm}

\usepackage[utf8]{inputenc} 
\usepackage{graphicx} 
\usepackage[small]{caption}
\usetikzlibrary{arrows,calc}
\numberwithin{equation}{section}
\usepackage{stmaryrd}
\usepackage{mathtools}

\usepackage{lmodern} 
\usepackage[T1]{fontenc} 

\usepackage{enumerate} 
\usepackage{verbatim} 
\usepackage{color}
\usepackage{xcolor}
\usepackage{tabularx}
\usepackage{xspace}

\usepackage[colorlinks, linkcolor=blue, citecolor = green!55!black]{hyperref}

\usepackage[letterpaper,nomarginpar]{geometry}

\usepackage[capitalise]{cleveref}

\numberwithin{equation}{section}

\newtheorem{maintheorem}{Theorem}

\newtheorem{theorem}{Theorem}[section]

\newtheorem*{theorem*}{Theorem}
\newtheorem{lemma}[theorem]{Lemma}

\newtheorem{proposition}[theorem]{Proposition}

\newtheorem{corollary}[theorem]{Corollary}

\theoremstyle{definition}{

\newtheorem*{definition*}{Definition}

\newtheorem*{question*}{Question}
\newtheorem*{example*}{Example}
\newtheorem*{examples*}{Examples}
\newtheorem{remark}[theorem]{Remark}
\newtheorem*{remark*}{Remark}

\newtheorem*{claim*}{Claim}
}

\newcommand{\R}{\mathbb{R}}
\newcommand{\E}{\mathbb{E}}

\newcommand{\N}{\mathbb{N}}

\newcommand{\C}{\mathbb{C}}
\renewcommand{\S}{\mathbb{S}}

\newcommand{\cC}{\mathcal{C}}
\newcommand{\cI}{\mathcal{I}}

\newcommand{\cB}{\mathcal B}

\newcommand{\cE}{\mathcal{E}}

\newcommand{\cN}{\mathcal{N}}
\newcommand{\cM}{\mathcal{M}}

\renewcommand{\P}{\mathbb{P}}
\newcommand{\bP}{\mathbf P}
\newcommand{\bM}{\mathbf M}
\newcommand{\cR}{\mathcal{R}}

\newcommand{\ep}{\varepsilon}
\renewcommand{\d}{\mathrm{d}}
\newcommand{\Cov}{\operatorname{Cov}}
\newcommand{\Var}{\operatorname{Var}}
\newcommand{\li}{\operatorname{Li}}

\newcommand{\id}{\operatorname{Id}}

\newcommand{\ind}[1]{\mathbbm{1}_{\{ #1\}}}

\newcommand{\one}{\mathbbm{1}}
\newcommand{\given}{\;\big|\;}

\newcommand{\cut}{\mathrm{cut}}
\newcommand{\fro}{\mathrm{froze}}

\renewcommand{\index}{\mathrm{ind}}

\newcommand{\caseS}{\hyperref[case:S]{\textup{(S)}}\xspace}
\newcommand{\caseR}{\hyperref[case:R]{\textup{(R)}}\xspace}

\usepackage{microtype}
\usepackage[backend = biber, sorting = nyt, defernumbers=true, style = numeric, isbn = false, url = false]{biblatex}
\bibliography{LitDirichlet}

\title[Multiplicative chaos from random multiplicative functions]{Multiplicative chaos from random multiplicative functions}
\author[H. Hu]{Haoran Hu}
\address{H.\ Hu\hfill\break
Department of Mathematics\\ Northwestern University\\ 2033 Sheridan Road\\ Evanston, IL 60208, USA.}
\email{haoranhu2030@u.northwestern.edu}
\author[Y. H. Kim]{Yujin H. Kim}
\address{Y.\ H.\ Kim\hfill\break
The Division of Physics, Mathematics and Astronomy \\ 
California Institute of Technology\\
1201 E California Blvd \\ Pasadena, CA 91125, USA.}
\email{yujin@caltech.edu}
 \author[X. Kriechbaum]{Xaver Kriechbaum}
\address{X.\ Kriechbaum\hfill\break
Institut de Math\'{e}matiques de Toulouse\\
Universit\'{e} de Toulouse\\
118 route de Narbonne\\
31062 Toulouse Cedex 9, France}
\email{xaver.kriechbaum@math.univ-toulouse.fr}

\allowdisplaybreaks
\begin{document}

\begin{abstract}

We construct the multiplicative chaos measures emerging from the  Dirichlet series of the twisted Steinhaus multiplicative function. Our work has three main features. First, our argument  treats the full subcritical and critical phases simultaneously. Second, we show that the same multiplicative chaos measure arises from any  approximation to the Dirichlet series in which the prime cutoff is sent to $\infty$ and the critical line is approached from the right at arbitrary speeds. Such results are known as ``universality'' results, and we show this via a novel and quick partial summation trick. Third, we show that the multiplicative chaos measure is mutually absolutely continuous with a GMC measure in the same phase, coupled on the same probability space, almost surely. Our work unifies and extends results from a recent breakthrough trilogy of papers by Gorodetsky and Wong.
\end{abstract}

{\mbox{}
\maketitle
}
\vspace{-.5cm}

\section{Introduction}\label{sec:intro}
This article is concerned with the study of multiplicative chaos measures emerging from probabilistic models for a wide class of twisted Dirichlet series, such as the Riemann $\zeta$ function. 
The study of the statistical properties of Dirichlet series has a long history, dating back to the work of Bohr and Jessen from the 1930s \cite{BohrJessenI,BohrJessenII}; see \cite{sound-survey} for a nice survey.
The work here fits in a  line of research initiated by the predictions of Fyodorov, Hiary, and Keating \cite{fyodorov-keating,fhk} that these objects (specifically, they study the logarithm of the $\zeta$ function) are related to log-correlated fields. Indeed, many of their predictions for the $\zeta$ function have now been verified \cite{abr1,abr2,paquette-zeitouni}. It is also closely related to the phenomenon of better than squareroot cancellation in sums of multiplicative functions. In some of these settings, results for random models can be transferred over to the arithmetic objects themselves, making the study of these models an important intermediate step: see, e.g., \cite{granville2001large,SW20,harper2023typical}, as well as the ICM survey of Harper \cite{harper-icm} where this point is made explicitly.

We proceed by introducing our main object of study. In Section~\ref{subsec:background}, we provide  background for the study of random models of Dirichlet series and of multiplicative chaos that we hope will be motivating to both probabilists and number theorists.
Our main results are stated in \cref{subsec:results}. 

\subsection{Set-up}\label{subsec:setup}
Let $\alpha: \N\to \C$ denote the Steinhaus random multiplicative function: that is, $\alpha(nm) = \alpha(n) \alpha(m)$ for any pair $n,m \in \N$, and for each prime $p$,  $\alpha(p) = \alpha_p$, where $(\alpha_p)_{p}$ denotes a sequence of  iid random variables uniformly distributed on the unit circle $\S^1\subset \C$. Here and below, $p$ as an index always ranges through the primes.
(See Remark~\ref{rk:rademacher} for comments on the Rademacher case.)

Let  $\bM$ denote the set of \emph{multiplicative functions} $f:\N\mapsto\C$: that is, $f$ is in $\bM$ if and only if $f(nm)=f(n)f(m)$ for any pair of \emph{coprime} integers $n,m\in\N$. For  $\theta>0$, let $\bP_\theta\subset\bM$ denote the set of multiplicative functions $g$ satisfying the following:
\begin{equation}\label{eq:PNT}
    \sum_{p\leq x}g(p)=\theta\frac{x}{\log x}+\cE_g(x)\,,
\end{equation}
where $\cE_g(x)=o(\li(x))$, for $\li(x) \coloneq \int_2^x \frac{\d t}{\log t} \sim x/\log x$, 
and satisfies the integrability condition 
\begin{equation}\label{eq:remainder}
  \int_{2-}^{\infty}\frac{\cE_g(x)}{x^2}~d x<\infty.
\end{equation}
Recall the prime number theorem states $f \equiv 1$ is in $\mathbf{P}_1$.

This article is concerned with the \emph{random Euler product/Dirichlet series}: 
\begin{equation}\label{eq:TruncatedEP}
    D_N(\alpha f,s)\coloneqq\prod_{p\leq N}\Big(1+\sum_{j\geq 1}\frac{\alpha_p^j f(p^j)}{p^{js}}\Big)=\sum_{\substack{n\geq1 \\ p\vert n \Rightarrow p\leq N}}\frac{\alpha(n)f(n)}{n^{s}}\,,
\end{equation}
near the critical line $s = \frac12 + \ep+ ix$, where the \emph{twist} function $f$ satisfies $|f|^2\in\mathbf{P}_{\theta}$ for $\theta \in (0,1]$ (the range $(0,1)$ for $\theta$ is called the \emph{subcritical phase}, while $\theta = 1$ is the \emph{critical phase}). 
When $f\equiv 1$, $D_N$ is known as the random model of the Riemann $\zeta$ function. 
Examples of non-constant twists $f$ of arithmetic interest include the indicator function of sums of two squares as well as divisor functions. 
On the other hand, as explained below, our treatment of $D_N(\alpha f, \cdot)$ will not use any particular arithmetic properties of $f$, demonstrating that our results are mainly driven by the multiplicative structure of $f$ and the log-correlated structure of $D_N$.  

\subsection{Background and motivation} \label{subsec:background}

\subsubsection{Log-correlated fields and multiplicative chaos}  \label{subsubsec:probability-motivation}
The random Dirichlet series $D_N$ is of probabilistic interest because as $\ep \to 0$ and $N \to \infty$, the quantity $|\log D_N(\alpha f, \frac12 +\ep+ix)|$ exhibits the structure of a \emph{logarithmically-correlated field}:  its covariance function (in $x$) blows up logarithmically approaching the diagonal, see the discussion above \cref{prop:error}.
There is a rich theory behind such fields, which are predicted to form a universality class. In particular, exponentiating such fields (in this case, the quantity $D_N$ itself) is predicted to lead to a random measure called \emph{multiplicative chaos}, an important object which has featured prominently in fields as diverse as random surfaces, random matrices, branching processes, turbulence, and finance.
Giving meaning to the aforementioned exponentiation is non-trivial and has been the subject of research since the pioneering work of Kahane '85 \cite{kahanegmc}. While the theory of multiplicative chaos measures associated to Gaussian log-correlated fields (Gaussian multiplicative chaos, or GMC) is well-developed, the theory for non-Gaussian log-correlated fields is rather nascent. The random Dirichlet series form an intriguing subset of non-Gaussian models not only because they come from number theory, but also because, as discussed below, their universality results appear to be the first of their kind.

We now give some background on the theory of multiplicative chaos measures, focusing on the aspects relevant to the results of this article.

\textit{The Gaussian story.}
We start with a brief, informal description of GMC on $\R^d$. We refer to the survey \cite{RV14} or the more recent textbook \cite{berestycki-powell} for more details. 
Consider a Gaussian log-correlated field $h$, which can be formally understood as a Gaussian field on $\R^d$ satisfying
\[
	\E[h(x)] = 0\,, \qquad\qquad \E[h(x) h(y)] = 2d \log(|x-y|^{-1}) + g(x,y)\,,
\]
where $g: \R^d \times \R^d \to \R$ is a continuous bounded function.
While $h$ cannot be realized as a bona fide random function due to the logarithmic singularity of its covariance, it can be understood as a Gaussian generalized function. 
One then gives a rigorous interpretation to the measure 
\begin{align}\label{def:gmc-informal}
	M_{\theta}(\d x) \ ``\coloneqq" \ \frac{e^{\theta h(x)}}{\E[e^{\theta h(x)}]} \d x \,,
\end{align}
called the GMC measure with parameter $\theta > 0$ associated to $h$, 
by fixing a sequence of continuous functions $(h_{\ep})_{\ep >0}$ approximating $h$, replacing $h$ with $h_\ep$ in the right-hand side above, and then taking a limit as $\ep \to 0$. 
This was first carried out in the seminal work of Kahane \cite{kahanegmc} under an additional assumption on the structure of $h$, who showed that the limiting measure is nontrivial for $\theta \in (0,1)$, the subcritical phase of GMC, and is equal to $0$ for $\theta \geq 1$.

\textit{Universality.} A cornerstone of the GMC theory is that any suitably regular approximation $(h_\ep)_{\ep >0}$ to $h$ should yield the same limiting measure $M_{\theta}$, thereby justifying the interpretation \eqref{def:gmc-informal} and establishing that $M_{\theta}$ is intrinsic to $h$. 
This result is sometimes referred to as \emph{universality}, and was established across several works since Kahane \cite{10.1214/09-AOP490,shamov2016gaussian,BerestyckiSimplePath}. 

\textit{Criticality.} There is a fascinating theory at the critical point $\theta_c = 1$, for which a non-trivial, non-atomic measure can be constructed via various renormalization procedures that all lead to the same limiting measure \cite{DRSV14a,DRSV14b,junnila-saksman,powell18,APS19,PowellCrit}. We refer to the survey of Powell \cite{PowellCrit} and also the article of Lacoin \cite{lacoin22} for a detailed account of critical GMC. In this article, we are particularly interested in the \emph{Seneta-Heyde} renormalization:
\[
	\sqrt{\Var(h_\ep)} \frac{e^{\theta_c h_\ep(x)}}{\E[e^{\theta_c h_\ep(x)}]} \d x \to M_{\theta_c}(\d x)\,.
\]
In \cite{junnila-saksman}, universality for critical GMC is shown for a wide class of approximations of $h$.

\textit{Non-Gaussian multiplicative chaos.}
The theory of multiplicative chaos associated to non-Gaussian log-correlated fields is much less developed. For instance, the aforementioned GMC universality results rely crucially on Gaussian tools such as Karhunen-Lo\`eve decomposition, Kahane's convexity inequality, and Cameron--Martin--Girsanov identities. 
It is stated in the recent work \cite{goro-wong-2} that their results on the multiplicative chaos associated to random Dirichlet series appear to give the first universality results for non-Gaussian multiplicative chaos.
Moreover, there are few constructions of critical multiplicative chaos measures associated to log-correlated fields which are not asymptotically Gaussian. Such constructions are generally highly technical and require special features of the underlying field, see, e.g., \cite{jego21}. We refer to \cite[Section~1.2]{KK24} for a summary of recent results on non-Gaussian multiplicative chaos and \cite[Section~1.4]{goro-wong-3} for a summary of critical constructions in particular. 
Regarding the random Dirichlet series, we explain below Eq.\ \eqref{def:SNep} that 
\begin{align}\label{eq:dn-approx}
    \log D_N\big(\alpha f, \tfrac12+\ep+ix\big) \approx \sum_{p\leq N} \frac{f(p)}{p^{\frac12+\ep}}\alpha_p p^{-ix}\,.
\end{align}
The right-hand side takes the form of a random Fourier series,  and the parameters $\ep$ and $N$ serve as two different forms of approximation to the full log-correlated field $\log D_{\infty}(\alpha f, \frac12+ix)$: $\ep$ is akin to mollification by a Poisson kernel, while $N$ is akin to mollification by a Dirichlet kernel. Our universality result will involve sending $N$ and $1/\ep$ to $\infty$ at any speed, implying that either mollification, or arbitrary combinations of the mollifications, all result in the same limiting measure.

\subsubsection{Random models of multiplicative functions and Dirichlet series}\label{subsubsec:numbertheory-motivation}
Understanding the size of sums of arithmetic multiplicative functions  is a central problem in analytic number theory. A convincing example comes from the M\"obius function $\mu: \N\to\R$, which we recall is defined by $\mu(n) = 0$ if $n$ has any repeated prime factors, and otherwise $\mu(n) = 1$ if $n$ has an even number of prime factors and $-1$ if $n$ has an odd number of prime factors.  
The statement $|\sum_{n\leq x} \mu(n)| = o(x)$ as $x\to\infty$ is equivalent to the prime number theorem, while the statement $|\sum_{n\leq x} \mu(n)| \leq x^{1/2+o(1)}$ is equivalent to the Riemann hypothesis. 
Other similarly-important examples of arithmetic multiplicative functions $a:\N\to\C$ are the Dirichlet characters $a(n) = \chi(n)$, related to $L$-functions, and the continuous characters $a(n)= n^{-it}$, related to vertical shifts of Dirichlet series.

More generally, it is of interest to study \emph{twisted} sums: i.e., $\sum_{n\leq x} a(n)f(n)$, where $f\in \mathbf{M}$. The twist function $f$ can be chosen to weight integers $n$ based on certain arithmetic criteria, such as being a sum of two squares, the number of prime factors $n$ has, etc. Morevoer, we shall see that for twists $f$ which are multiplicative but do not satisfy any particular arithmetic properties, the sum $\sum_{n\leq x} a(n) f(n)$ exhibits  similar behavior to $\sum_{n\leq x} a(n)$, indicating that the underlying behavior of the sums is largely driven by multiplicativity. 

The connection between sums of $a(n) f(n)$ and their (truncated) Dirichlet series $D_N(af, s)$, defined in \eqref{eq:TruncatedEP}, can be made through Plancherel's identity \cite[Theorem~5.4]{montgomery-vaughan-1}: for any $\ep >0$, 
\begin{align}\label{eq:parseval}
    \int_0^{\infty} \frac{1}{x^{1+\ep}} \ \bigg|\!\!\!\sum_{\substack{n \leq x \\ p | n \Rightarrow p \leq N}} a(n)f(n) \bigg|^2 \d x  \ = \ \frac{1}{2\pi} \int_\R \frac{|D_N(af,\frac12 + \frac{\ep}{2}+ix)|^2}{|\frac12 + \frac{\ep}{2} + ix|^2} \d x\,.
\end{align}

Part of what makes the study of sums of $a(n)$ challenging is that the functions $a$ of arithmetic interest are often oscillatory, so that their sums $\sum_{n \leq x} a(n)$  exhibit non-trivial cancellations. An analogy is often made with sums of independent random variables: if each $a(n)$ was an independent, mean $0$, variance $1$ random variable, then according to the central limit theorem, the typical size of $\sum_{n\leq x} a(n)$ is $\sqrt{x}$. 
However, this independent model of $a(n)$ completely  ignores the dependence structure inherited from multiplicativity, and also fails to explain, for instance, the phenomenon of \emph{better than squareroot cancellation},  see \cite{harper-icm} for a detailed survey of this. 

By contrast, consider the Steinhaus random multiplicative function (RMF) $\alpha(n)$, which models the  complex-valued $\chi(n)$ and $n^{-it}$, and the Rademacher RMF $r(n)$, which models the M\"obius function and is defined by assigning iid Rademacher random variables to each $r(p)$, setting $r(n) = 0$ if $n$ is not squarefree, and defining $r$ multiplicatively otherwise. These RMFs, introduced by Wintner in the 1940s \cite{wintner}, assign iid values to the primes but still have multiplicative dependency. 
There are numerous ways in which these RMFs model their deterministic counterparts: we refer to \cite[Section~3]{harper-icm} for motivation for the validity of the RMF as a model, and to \cite{granville2001large,harper2023typical} for rigorous results in which establishing the results for RMFs first was a crucial intermediate step.

The random model $D_N(\alpha f, \cdot)$ arises naturally in a different way. For simplicity, consider the random model of $\zeta$, obtained from taking $f\equiv 1$. For $\sigma >1$, we have the Euler product representation of $\zeta(\sigma +it)$. Taking logarithms and using Taylor's theorem yields
\[
    \log \zeta(\sigma + it) \approx \sum_p \frac{p^{-it}}{p^\sigma}\,.
\]
Now, consider $t = \tau+ x$, for $\tau \sim \mathrm{Unif}[0,T]$, $T$ large, and $x$ lying in a compact interval. Due to the linear independence of $(\log p)_p$, one can see $(p^{-i\tau})_p$ converges as $T\to\infty$ to iid Steinhaus random variables (uniform on $\S^1 \subset \C$) in the sense of finite-dimensional distributions. Thus, for $\sigma >1$,
\[
    \log \zeta(\sigma + i(\tau+x)) \approx  \sum_p \frac{\alpha(p)}{p^\sigma}p^{-ix}\,.
\]
In \cite{SW20}, Saksman and Webb were able to show that for $\sigma =1/2$, the left-hand side converges in distribution as $T\to\infty$ to the right-hand side as certain generalized functions in $x$. This extends earlier work of Bagchi \cite{bagchi81} and Bohr--Jessen \cite{BohrJessenI,BohrJessenII}, showing the convergence when $\sigma > 1/2$.

\textit{Sums of RMFs and multiplicative chaos.}
In his seminal work, Harper 
\cite{harper2020moments} showed 
\begin{equation*}
    \E\Big|\sum_{n\leq N}\alpha(n)\Big|^{2q}\asymp\Big(\frac{N}{1+(1-q)\sqrt{\log\log N}}\Big)^q,\quad \forall q\in[0,1].
\end{equation*}
Taking $q=1/2$, one observes better than squareroot cancellation,  confirming a conjecture of Helson \cite{helson2010hankel} (we note that a simpler proof of Helson's conjecture was given in \cite{gorodetsky2025short}, though with a weaker result). 
The proof relies on the observation that the above moments are related to the critical multiplicative chaos associated to $D_N(\alpha,\frac12+ix)$, as can be seen from \eqref{eq:parseval}.

In a breakthrough trilogy of work by Gorodetsky and Wong \cite{goro-wong-1,goro-wong-2,goro-wong-3}, the result of Harper in the previous display is greatly improved. Gorodetsky and Wong establish the limiting distribution of $\sum_{n \leq N} \alpha(N) f(N)$ for $f$ lying in a certain subset of $\mathbf{P}_{\theta}$, $\theta \in (0,1]$, and in particular  obtain an exact first-order asymptotic for the moments of this sum.
Both the limiting distribution and the moment asymptotic
are expressed in terms of the multiplicative chaos measure associated to $D_N$.

\subsection{Discussion of main results}
\label{subsec:results}
Our work in the present article is greatly inspired by the aforementioned  trilogy of work by Gorodetsky--Wong.
Their first two papers  \cite{goro-wong-1, goro-wong-2} treat the so-called $L^2$ and $L^1$ regimes of the subcritical phase $\theta \in (0,1)$ respectively, while their third paper \cite{goro-wong-3} treats the critical phase $\theta =1$.
Each paper can be roughly divided into two parts. In each ``part 1'', the authors construct the multiplicative chaos 
arising from $D_N(\alpha f,\frac12+\ep_N+ix)$, for $f$ in a subset of $\bP_{\theta}$, and moreover show a universality result: the limiting multiplicative chaos measure is independent of the sequence $N,1/\ep_N\to \infty$ for a certain range of speeds $\ep_N\to 0$ (see the discussion below \eqref{eq:dn-approx} as to why this should be thought of as universality).
In each ``part 2'', the authors develop a martingale central limit theorem that allows them to connect  the limiting distribution of the random sums $\sum_{n \leq N} \alpha(n) f(n)$, normalized by an explicit function of $N$ and $f$, 
with the multiplicative chaos measure associated to $D_N(\alpha f,\frac12+\ep_N+ix)$  for a particular speed $\ep_N \to 0$. This is related to \eqref{eq:parseval}. Together, parts 1 and 2 allow for the determination of the limiting distribution of $\sum_{n\leq N} \alpha(n) f(n)$ as $N\to\infty$.

Our work here solely addresses  the ``part 1'' results of their work; we do not address the ``part 2'' results at all, other than to note that they each require a universality result from ``part 1'' in order to deduce the limiting distribution as $N,1/\ep_N \to \infty$.
Towards part 1, the authors develop different, interesting methods in each of \cite{goro-wong-1, goro-wong-2, goro-wong-3} to handle the unique challenges stemming from the lack of Gaussianity and the lack of integrability present in the $L^1$ and critical phases. 
Further, they mention in those papers that their universality results appear to be the first universality statements in the theory of non-Gaussian multiplicative chaos.

In this article, we present a unified proof of all  ``part 1'' results of \cite{goro-wong-1,goro-wong-2,goro-wong-3}, agnostic to whether the phase is critical or subcritical.
Further, our results expand the notion of universality for the limiting chaos measure and also enlarge the subset of twists in $\mathbf{P}_{\theta}$  for which this measure can be constructed throughout $\theta\in(0,1]$. 
Lastly, we show that the properties of this measure $\mu_\infty$ are in fact governed by GMC, in the sense that $\mu_\infty$ is mutually absolutely continuous with a GMC measure, almost surely.
This immediately yields new, fundamental properties of $\mu_\infty$, see \cref{rk:gmc-properties}.
We make a detailed comparison between our results and  those of \cite{goro-wong-1,goro-wong-2,goro-wong-3}
in \cref{rk:comparing-results}.

\begin{maintheorem}\label{thm:main}
    Let $\theta\in (0,1]$, and let $f: \N \to \C$ satisfy
    \begin{equation}\label{eq:SuffTwistCond}
     |f|^2\in\bP_\theta\,,\qquad
    \sum_{p\geq2}\bigg(
    \frac{|f(p)|^3}{p^{3/2}} \log^{2+\delta_1} p+
    \frac{|f(p^2)|^2}{p^2}\log^{1+\delta_2} p+\sum_{j\geq3}\frac{\vert f(p^j)\vert}{p^{j/2}}\bigg)<\infty
    \end{equation}
    for some $\delta_1, \delta_2 >0$.
    For each $\varepsilon\geq0$ and $N\geq0$, we set 
    \[
    V_{N,\varepsilon}(\theta) \coloneqq \one_{\theta\in (0,1)}+\one_{\theta=1}\sum_{p=1}^N \frac{|f(p)|^2}{p^{1+2\varepsilon}}\qquad\text{and}\qquad \mathrm{d}\rho_{N,\varepsilon}(x) \coloneqq \sqrt{V_{N,\varepsilon}(\theta)} \ \d x\,.
    \] 
    We define the random measure $\mu_{N,\varepsilon}$ on $\R$ via
    \begin{equation}\label{eq:mu}
        \mu_{N,\varepsilon}(\d x)\coloneqq\frac{\vert D_{N}(\alpha f,\frac12+\varepsilon+ix)\vert^{2}}{\E[\vert  D_{N}(\alpha f,\frac12+\varepsilon+ix)\vert^{2}]} \,\d\rho_{N,\varepsilon}(x) \,.
    \end{equation}
    There exists a nontrivial, non-atomic random measure $\mu_\infty$ on $\R$ such that the following hold.
    \begin{enumerate}[(i)]
        \item In the subcritical phase $\theta\in(0,1)$,
        for every $r\in[1,1/\theta)$ and $h\in C_c(\R)$,
    \begin{equation}
        \mu_{N,\varepsilon}(h)\xrightarrow[\min(N,\,1/\ep)\to\infty]{L^r}\mu_\infty(h).
    \end{equation}
        \item In the critical phase $\theta=1$, for every $h\in C_c(\R)$,  
    \end{enumerate}
    \begin{equation}\label{eq:critical-convergence}
        \mu_{N,\varepsilon}(h)\xrightarrow[\min(N,\, 1/\ep)\to\infty]{\P}\mu_\infty(h).
    \end{equation}
    Moreover, for all $\theta \in (0,1]$, there exists a GMC measure $\mu_{\infty}^{\mathrm{GMC}}$ of inverse temperature $\theta$ coupled on the same probability space as $\mu_{\infty}$ such that the two measures are mutually absolutely continuous, almost surely. The Radon-Nikodym derivative is continuous and  nonnegative, is positive $\mu_{\infty}^{\text{GMC}}$-almost everywhere, and its $L^\infty$-norm on any compact interval has finite positive moments.
\end{maintheorem}
We also believe our method could be used to treat the case where $\alpha(n)$ is replaced by the Rademacher RMF: see \cref{rk:rademacher}.
We outline the main ideas of the proof of \cref{thm:main} in Section~\ref{sec:proof-main}. As explained there, our Gaussian approximation method yields a coupled Gaussian field $G_{N,\ep}$  such that $| D_N(\alpha f, \frac12+\ep+ix)|^2$ is well-approximated by $\exp(2\Re G_{N,\ep}(x))$ uniformly over $N, \ep \geq 0$ and $x$ in any fixed compact interval.

\begin{remark}\label{rk:gmc-properties} 
In addition to nontriviality and non-atomicity, several properties of   $\mu_\infty$ for $\theta \in (0,1]$ follow as immediate consequences of absolute continuity with $\mu_\infty^{\text{GMC}}$ with continuous Radon-Nikodym derivative, such as the Hausdorff dimension of its support set and its multifractal spectrum.
\end{remark}

\begin{remark}
While $V_{N,\ep}(\theta) = 1$ for $\theta \in (0,1)$, the quantity $V_{N,\ep}(1)$ satisfies
\[
    \sup_{N,\ep\geq 0} \big|V_{N,\ep}(1) - \log\log\min\big(N,e^{1/\ep}\big)\big| \leq C
\]
for some constant $C>0$. This is a consequence of \eqref{eq:criteria1}.
\end{remark}

\begin{remark}
We actually prove \cref{thm:main} under the following weaker twist condition:
\begin{equation}\label{eq:TwistCondition}
    |f|^2\in\bP_\theta,\quad
    \sum_{n =1}^{\infty} \bigg(\sum_{e^n < p\leq e^{n+1}} \frac{|f(p)|^3}{p^{3/2}}  \bigg)^{1/3}
    + \sum_{n =1}^{\infty} \bigg(\sum_{e^n < p\leq e^{n+1}} \frac{|f(p^2)|^2}{p^2}\bigg)^{1/2}
    +
    \sum_{p\geq2}\sum_{j\geq3}\frac{\vert f(p^j)\vert}{p^{j/2}}<\infty\,.
    \end{equation}
The fact that  \eqref{eq:TwistCondition}  is implied by \eqref{eq:SuffTwistCond} can be seen from H\"older's inequality: for any $\delta_1>0$, 
\begin{align*}
\sum_{n=1}^\infty &\bigg(\sum_{e^n<p\le e^{n+1}} \frac{|f(p)|^3}{p^{3/2}}\bigg)^{1/3} = \sum_{n=1}^\infty  n^{-\frac{2+\delta_1}{3}}\bigg(\sum_{e^n<p\le e^{n+1}} n^{2+\delta_1}\frac{|f(p)|^3}{p^{3/2}}\bigg)^{1/3}\\
&\le \bigg(\sum_{n=1}^\infty \sum_{e^n<p\le e^{n+1}} n^{2+\delta_1} \frac{|f(p)|^3}{p^{3/2}}\bigg)^{1/3}\cdot\bigg(\sum_{n=1}^\infty n^{-\frac{2+\delta_1}{3}\cdot \frac{3}{2}}\bigg)^{\frac{2}{3}} \leq C \bigg(\sum_{p\ge 2} \frac{|f(p)|^3}{p^{3/2}}\log^{2+\delta_1} p\bigg)^{1/3}\,,
\end{align*}
for some  $C>0$.
Similarly, for any $\delta_2>0$, inserting $n^{1+\delta_2}$ and applying H\"older's inequality yields
\begin{align*}
    \sum_{n=1}^\infty \bigg(\sum_{e^n<p\le e^{n+1}} \frac{|f(p^2)|^2}{p^2}\bigg)^{1/2} 
   \leq C \bigg(\sum_{p\ge 2} \frac{|f(p^2)|^2}{p^2}\log^{1+\delta_2} p\bigg)\,.
\end{align*}
\end{remark}

\begin{remark}[Comparison with the work of Gorodetsky and Wong]
\label{rk:comparing-results}
Here, we explain the relationship between our results
and those of the closely-related work of Gorodetsky--Wong.
Our results correspond to \cite[Theorem~2.5]{goro-wong-1}, \cite[Theorem~1.1]{goro-wong-2}, and \cite[Theorem~1.1]{goro-wong-3}: that is, their ``part 1'' results, as explained at the start of this subsection. 

Firstly, our result identifies the limiting measure $\mu_\infty$  in terms of GMC and the field $D_N$ in terms of the exponential of a Gaussian log-correlated field for all $N\geq 1$, $\theta \in (0,1]$, and $f$ satisfying \eqref{eq:TwistCondition}. These descriptions  were previously absent in the literature aside from the untwisted case $f\equiv 1$ (and thus $\theta =1$), where these results were shown by Saksman--Webb \cite{SW20}. 

Regarding the class of twists $|f|^2\in \mathbf{P}_\theta$ that are treated, it is easiest to compare the twist conditions of Gorodetsky--Wong with our \eqref{eq:SuffTwistCond} (recall \eqref{eq:TwistCondition} is more general). 
For the critical phase $\theta =1$, \cite[Theorem~1.1]{goro-wong-3} is obtained for the untwisted model $f\equiv 1$ only.
They develop a method that shows the following: given a construction of $\mu_\infty$ as $N\to\infty$ for $\ep = 0$ and $f \in \mathbf{P}_1$ satisfying\footnote{
They remark in the second display of \cite[Section~2]{goro-wong-3} that their results hold for twists $f \in\mathbf{P}_1$ satisfying 
a weaker twist condition than \eqref{eq:GWTwist}, but it was confirmed to us in \cite{email} that this is a typographical error, and \eqref{eq:GWTwist} is indeed what they require. Observe $f\equiv 1$ does not satisfy the twist condition in \cite[Section~2]{goro-wong-3} since $\log^2 p/p$ is not summable.}
\begin{equation}\label{eq:GWTwist}
\sum_{p\ge 2} \bigg(\frac{|f(p)|^3}{p^{3/2}}\log^3 p+\frac{|f(p^2)|^2}{p^2}\log^2 p+\sum_{j\ge 3}\frac{|f(p^j)|}{p^{j/2}}\log^2 p\bigg)<\infty\,,
\end{equation}
they show one obtains the same limit for a class of speeds $\ep_N \to 0$  as $N\to \infty$ (a universality result elaborated on in the next paragraph).
However, before our article, the $\ep =0$ construction of $\mu_\infty$ was only done in the case $f\equiv 1$, by Saksman--Webb \cite{SW20}.
In the $L^1$ regime $\theta \in [1/2,1)$, \cite{goro-wong-2} requires their twists $f$ to satisfy  \eqref{eq:GWTwist},
so that our twist condition \eqref{eq:SuffTwistCond} features gains of various log powers on all terms. In the $L^2$ regime $\theta \in(0,1/2)$, their twist condition \cite[(2.17)]{goro-wong-1} is in fact weaker than  \eqref{eq:SuffTwistCond}. Their method takes advantage of the $L^2$-boundedness of $\mu_{N,\ep}$, which is not available for $\theta \geq 1/2$.

Regarding universality, we show convergence to the same limit as the prime cutoff $N$ tends to $\infty$ and the line $\Re s = 1/2+\ep$  approaches the critical line at an arbitrary joint speed: $\min (N,1/\ep)\to\infty$ for any $\theta \in (0,1]$. We also show $L^r$ convergence for the optimal range $r\in [1,1/\theta)$ for all $\theta \in (0,1)$, at any joint speed $\min (N,1/\ep)\to\infty$.
We describe in \cref{sec:proof-main} how our work builds on the work of \cite{SW20} to treat the $\ep = 0$ case, and then treats $\ep \geq 0$ via a  partial summation trick that will have future applications, see \cite{kim-kriechbaum-future}.
In the $L^1$ regime $\theta \in [1/2, 1)$,  \cite{goro-wong-2}  shows convergence of $\mu_{N,\ep}$ in $L^1$ as $\min (N,1/\ep)\to\infty$ and proves the full $L^r$ convergence when one of the two approximation parameters is fixed: either $N= \infty, 1/\ep \to \infty$ or $1/\ep = \infty, N\to\infty$. In the $L^2$ regime $\theta \in (0,1/2)$, \cite{goro-wong-1} shows $L^2$ convergence when one of the two approximation parameters is fixed, though the same methods of \cite{goro-wong-2} would allow for $\min (N,1/\ep)\to\infty$. 
The $\theta=1$ result of \cite{goro-wong-3} requires the off-critical parameter to be of the particular form $\ep \coloneqq \ep_N(u) = u/\log N$, with $u\geq 0$ fixed; however it was communicated to us \cite{email} that the same methods of \cite{goro-wong-2} could be used to allow for $\min(N,1/\ep)\to\infty$.
\end{remark}

\subsection{Notation and conventions} \label{subsec:notation} In what follows, we almost always work over a fixed compact interval $\cI \subset \R$. As such, we  write $\|f\|_{\infty}$ to denote the norm in $L^{\infty}(\cI)$, even when $f$ is defined on all of $\R$.  We use the standard (Bachmann-Landau) big O and little o notation. We also use Vinogradov's notation: $f(N) \ll g(N)$ means $f(N) = O(g(N))$ as $N\to\infty$.

\subsection*{Acknowledgments}
We heartily thank Ofir Gorodetsky and Mo Dick Wong for generous feedback on a draft of this article that in particular clarified \cref{rk:comparing-results}.
YHK was supported by NSF grant DMS-2502920. Much of this work was completed while YHK was a visitor at the Forschungsinstitut f\"ur Mathematik (FIM) at ETH Z\"urich and the Institut de Math\'ematiques de Toulouse (IMT) at the Universit\'e de Toulouse III Paul Sabatier. We thank CIMI LabEx for funding the visit to the IMT. YHK also thanks Sophia Loo for communicating to us \href{https://ems.press/content/serial-article-files/9616}{this story} about Banach, Nikodym, and Steinhaus, three mathematicians who feature prominently here.
XK was supported by the ANR MBAP-P (ANR-24-CE40-1833) project.

\section{The proof of Theorem~\ref{thm:main}}\label{sec:proof-main}
In this section, we explain the ideas behind the proof of Theorem~\ref{thm:main}, record the main inputs, and then prove \cref{thm:main} given these inputs. Our proof has three main parts: the isolation of the ``log-correlated part'' of $\log D_N$, a coupling with a Gaussian field on the critical line, and then a comparison with the Gaussian field arbitrarily to the right of the critical line.

Fix $f$ satisfying \eqref{eq:TwistCondition} throughout.
Our starting point is that 
\begin{equation}\label{def:SNep}
    S_{N,\varepsilon}(x)=\sum_{p\leq N}\frac{f(p)}{p^{\frac{1}{2}+\varepsilon}} \alpha_p \, p^{-ix}\,.
\end{equation}
forms the   ``log-correlated part of $\log D_N(\alpha f, \frac12+\ep+ix)$''
(note we do not literally study $\log D_N$ as zeros of $D_N$ may exist).
Motivation for $S_{N,\ep}$ capturing the singular properties of $\log D_N$ comes from isolating the $j=1$ term in the summation defining $D_N$ in \eqref{eq:TruncatedEP} as the dominant term. The field $S_{N,\ep}$ then emerges from  formally taking logarithms and using $\log (1+x) \approx x$. This heuristic is well-known and was already made rigorous in \cite{SW20} in the case $(f\equiv 1, \ep = 0)$
and in \cite{goro-wong-1,goro-wong-2,goro-wong-3} for a subclass of twists $|f|^2 \in \mathbf{P}_{\theta}$.
The analogous result for us is \cref{prop:error} below, which, compared to these other works, treats a larger class of twists and provides information on the zero set of $|D_N|^2/\exp(2\Re S_{N,\ep})$ needed for deducing the mutual absolute continuity part of \cref{thm:main}.
We prove \cref{prop:error}  in \cref{sec:error} using a new argument based on the blocking method described below \cref{prop:coupling}.

\begin{proposition}\label{prop:error}
Fix a compact interval $\cI\subset \R$ and $f\in\mathbf{M}$ satisfying \eqref{eq:TwistCondition}. As  $\min(N,1/\ep)\to\infty$, the random field 
$E_{N,\ep}^{(1)}(x) \coloneq |D_N(\alpha f, \frac12 + \ep + ix)|^2/\exp(2\Re S_{N,\ep}(x))$ converges almost surely to a continuous, nonnegative random function $E^{(1)}$, uniformly over  $x\in \mathcal{I}$. 
The zero set of $E^{(1)}$ is equal to the zero set of $D_{N_0}(\alpha f, \frac12+ix)$, 
where $N_0\in\N$ is a deterministic constant depending only on $f$. Moreover, for all $\lambda>0$, we have 
    \[
     \mathbb{E}\Big[\sup_{N, \ep \ge 0}\|E_{N,\ep}^{(1)}\|_{L^\infty(\cI)}^\lambda\Big]<\infty\,.
    \]

\end{proposition}

\cref{prop:error} implies that, to understand the measures $\mu_{N,\ep}$ defined in \eqref{eq:mu}, it suffices to understand $\exp(2\Re S_{N,\ep}(x)) \d \rho_{N,\ep}$. 
Our strategy for understanding this family of measures is to show that $S_{N,\ep}$ is nearly Gaussian, in the following sense. In the sequel, we write $\cN^{\C}(0,1)$ to denote the law of the random variable $(X+iY)/\sqrt2$, where $X$ and $Y$ are iid standard Gaussians in $\R$. We prove $S_{N,\ep}$ may be coupled to the Gaussian field $G_{N,\ep}$, defined by 
\begin{equation}\label{eq:Gau}
    G_{N,\varepsilon}(x)\coloneqq\sum_{p\leq N}\frac{f(p)}{p^{\frac12+\ep}}Z_p \, p^{-ix}\,,
\end{equation}
where the $Z_p \sim \cN^{\C}(0,1)$ are iid,
such that the difference $S_{N,\ep}- G_{N,\ep}$ again converges in probability to a continuous bounded function uniformly on any compact interval as $\min(N,1/\ep) \to \infty$. This is nice because \cref{thm:main} can then be deduced from the results of \emph{Gaussian} multiplicative chaos\footnote{Namely, convergence and universality for the GMC associated to $G_{N,\ep}$.}, which, as mentioned in \cref{sec:intro}, are standard and collected in \cref{prop:gmc} below. Thus, by developing a Gaussian coupling, we bypass the complications of universality and the $L^1$ and critical phases, which are generally highly non-trivial in non-Gaussian cases as outlined in detail in \cite{goro-wong-2,goro-wong-3}.

We now describe how we achieve a coupling of $S_{N,\ep}$ and $G_{N,\ep}$.
We start by treating the case \emph{on the critical line} ($\ep = 0$): we construct a coupling of $(Z_p)_p$ and $(\alpha_p)_p$ such that $S_{N,0}$ is close to $G_{N,0}$. 
\begin{proposition}\label{prop:coupling}
Fix a compact interval $\mathcal I\subset \R$.
For any $f\in \mathbf{M}$ satisfying \eqref{eq:TwistCondition}, there exists a coupling between the $\alpha_p$ and an iid sequence $Z_p \sim \cN^{\C}(0,1)$  
such that
the error function 
\[
  E_{N,0}(x) \coloneqq S_{N,0}(x) - G_{N,0}(x)  = \sum_{p\leq N} \frac{f(p)}{\sqrt{p}} (\alpha_p-Z_p) p^{-ix}
\]
converges almost surely in $L^{\infty}(\mathcal{I})$ to a random continuous function $E(x)$. Moreover, for each $\lambda\geq 0$, we have the exponential moment bound
\begin{align}\label{eq:couplingthm-exp-moment-N}
  \E\big[\exp\big(\lambda \sup_{N\geq 0} \|E_{N,0}\|_{L^\infty(\cI)}\big) \big]<\infty\,,
\end{align}
which in particular implies 
\begin{align}\label{eq:couplingthm-exp-moment}
  \E \big[\exp(\lambda\|E\|_{L^\infty(\cI)})\big] < \infty \,.
\end{align}
\end{proposition}
The $f\equiv 1$ case of \cref{prop:coupling} was shown by Saksman and Webb in \cite{SW20}. Our proof is inspired by their argument, which can be summarized as follows. First, they specify a deterministic sequence of scales $(r_m)_{m\in \N}$, $r_m \uparrow  \infty$ as $m \to\infty$, and consider ``blocks'' of the form 
\[
    Y_m(x) \coloneq \sum_{k= r_m}^{r_{m+1}-1} \frac{1}{\sqrt{p_k}} \alpha_{p_k} p_k^{-ix}\,,
\]
where $p_k$ denotes the $k$th smallest prime.
The sequence $(r_m)_{m\in\N}$ is chosen for the following.
\begin{enumerate}
    \item (Error from freezing) Defining a ``frozen'' version
    \[
        \widetilde{Y}_m(x) \coloneqq A_m  p_{r_m}^{-ix} \,, \qquad A_m \coloneqq \sum_{k= r_m}^{r_{m+1}-1} \frac{1}{\sqrt{p_k}} \alpha_{p_k}\,,
    \]
    the difference $\|Y_m - \widetilde{Y}_m\|_{L^{\infty}(\cI)}$ is summable in $m$. 
    \item (Error from coupling) For each $m\in \N$,  $A_m$ can be coupled with a complex Gaussian (being a sum of independent random variables), such that the difference between this Gaussian and the corresponding $A_m$ is summable in $m$.
\end{enumerate}

For general twists $f$ satisfying \eqref{eq:TwistCondition}, it is not possible to directly implement this strategy.
This is because $f$ can vanish, which makes their Gaussian coupling tool not applicable as written, and more seriously because the support of $f$ can be very sparse. As a result,  $r_{m+1}-r_m$ must grow very fast to make the coupling error small; however, this in turn makes the error from freezing large. 

We now describe how we overcome this, aiming for a high-level overview as the proofs are already very short.
Our construction proceeds by creating ``two levels'' of blocks. For each $n\in \N$, we define the \emph{scale} $I_n \coloneq \{p \text{ prime} : p \in (e^n, e^{n+1}]\}$ and a cutoff parameter $\cut_n$, defined in \eqref{def:cutn}. Within each scale, we group primes into blocks $B$ until the variance of the block 
first exceeds $\cut_n$. Note each scale contains at most one ``incomplete'' block $B$, having $V_B < \cut_n$. 
See \cref{subsec:block-coupling-scheme-proof-thm} for our full construction, as well as a detailed overview of how we bound the errors from freezing and coupling.
For now, we simply mention that $\cut_n$ is chosen so that the total error from freezing accumulated on $I_n$ is equal to the total error from coupling, and the $|f(p)|^3$ part of the twist condition \eqref{eq:TwistCondition} leads to these errors being summable over $n\in\N$. Our bound on the error from coupling is derived from a textbook Berry--Esseen theorem, which is the only Gaussian coupling input we need; see \cref{subsec:gaussian-coupling-tools}. Our bound on the error from freezing comes from subgaussian estimates on sums of elements of Hilbert spaces with iid coefficients, the theory of which dates back to Kahane \cite{KahaneBook}; see \cref{subsec:hilbert}.

We now turn our attention to bounding $S_{N,\ep} - G_{N,\ep}$ (i.e., the fields \emph{off the critical line}).
In \cite[Page~4]{goro-wong-2}, it is mentioned that a Gaussian coupling strategy may ``not work very well for the approximation away from the critical line'', since one would need to implement an uncountable collection of couplings indexed by $\ep\geq 0$. 
This sounds very reasonable, and so one of our contributions is a partial summation trick that allows us to quickly deduce the closeness of $S_{N,\ep}$ and $G_{N,\ep}$, uniformly over $\ep \geq 0$, from the  closeness of $S_{N,0}$ and $G_{N,0}$ (i.e., the fields \emph{on the critical line}). 
In other words, once we have a coupling on the critical line, we reduce the question of universality to partial summation and GMC universality.
This is \cref{prop:off-critical-error}, proved in \cref{sec:off-critical-errors}.

\begin{proposition}\label{prop:off-critical-error}
    Fix a compact interval $\cI\subset \R$. For any $f \in \mathbf{M}$ satisfying \eqref{eq:TwistCondition},
    as $\min(N,1/\ep)\to\infty$, 
    the error function
    \[
    E_{N,\varepsilon}(x)\coloneqq S_{N,\varepsilon}(x)-G_{N,\varepsilon}(x)
    \]
    converges almost surely in $L^{\infty}(\cI)$ to the random function $E(x)$ from \cref{prop:coupling}. In addition, 
    \[
    \E\Big[ \exp\Big(\lambda  \sup_{N,\varepsilon\geq0} \|E_{N,\varepsilon}\|_{L^\infty(\cI)}\Big) \Big] <\infty\,, \qquad \text{for each } \lambda >0\,.
    \]
\end{proposition}

Finally, we state the aforementioned GMC input corresponding to the Gaussian fields  $G_{N,\ep}$, and then show how this combines with \cref{prop:error} and \cref{prop:off-critical-error} to yield \cref{thm:main}. Define the measures
\begin{equation}\label{eq:GMC-measures}
    \mu_{N,\varepsilon}^{\text{GMC}}\coloneq\frac{e^{2\Re  G_{N,\varepsilon}(x)}}{\E e^{2\Re  G_{N,\varepsilon}(x)}} \d\rho_{N,\varepsilon}(x)\,,
\end{equation}
where we recall the reference measures $\d\rho_{N,\varepsilon}(x)$ from the statement of Theorem~\ref{thm:main}. The following proposition is proved in Appendix~\ref{sec:app}.
\begin{proposition}[Universality of GMC]\label{prop:gmc}
Fix a compact interval $\cI$ and $f\in\mathbf{M}$ satisfying \eqref{eq:TwistCondition}.
     In both the critical and subcritical phases, in the space $\mathcal{M}(\cI)$ of Radon measures on $\cI$ equipped with the topology of weak convergence,  the sequence of measures $\mu_{N,\ep}^{\text{GMC}}$ converges in probability as $\min(N,1/\varepsilon)\to\infty$ to a non-trivial, non-atomic measure $\mu_{\infty}^{\text{GMC}}$.
    In the subcritical phase $\theta\in(0,1)$, $\mu_{N,\varepsilon}^{\text{GMC}}(\cI)$ converges in $L^r$ as $\min(N,1/\varepsilon)\to\infty$ to $\mu_{\infty}^{\text{GMC}}(\cI)$ for all $r\in[1,\frac1\theta)$.
\end{proposition}

\begin{proof}[Proof of Theorem~\ref{thm:main}]
We first show convergence in probability of $\mu_{N,\ep}$ to $\mu_{\infty}$ for $\theta \in (0,1]$. 
Fix $h \in C_c(\R)$, and let the support of $h$ be contained in some compact interval  $\mathcal{I} \subset \R$.  
We use the following consequence of the continuous mapping theorem for convergence of measures (see also \cite[Lemma~2.10]{goro-wong-2}): if $X_n \xrightarrow{\P}X$ in the space of continuous functions $C(\mathcal{I})$ and $\nu_n \xrightarrow{\P} \nu$ in  $\mathcal{M}(\cI)$, then $X_n \d \nu_n \xrightarrow{\P} X \d \nu$ in $\mathcal{M}(\cI)$.
Define
    \begin{equation}\label{def:xne}
        X_{N,\varepsilon}(x) \coloneq \frac{\mathrm{d}\mu_{N,\varepsilon}(x)}{\mathrm{d}\mu_{N,\varepsilon}^{\text{GMC}}(x)} = E_{N,\ep}^{(1)}(x)\, e^{2\Re E_{N,\varepsilon}(x)} \, \frac{\mathbb{E}[e^{2\Re G_{N,\varepsilon}(x)}]}{\mathbb{E}[|D_N(\alpha f, \frac{1}{2}+\varepsilon+ix)|^2]}\,.
    \end{equation}
Since $\d \mu_{N,\ep} = X_{N,\ep} \d \mu_{N,\ep}^{\mathrm{GMC}}$, and since $E_{N,\ep}^{(1)}$, $E_{N,\ep}$, and $\mu_{N,\ep}^{\mathrm{GMC}}$ all converge in probability due to Propositions~\ref{prop:error},~\ref{prop:off-critical-error}, and~\ref{prop:gmc}, convergence in probability of $\mu_{N,\ep}$ to $\mu_\infty$ follows once we show
    \begin{align}\label{eq:proof-thm1-pfratio}
        \lim_{\min(N,1/\ep)\to \infty}\frac{\mathbb{E}[e^{2\Re G_{N,\varepsilon}(x)}]}{\mathbb{E}[|D_N(\alpha f, \frac{1}{2}+\varepsilon+ix)|^2]} = \mathfrak{c}
    \end{align}
    uniformly over $x\in \cI$, for some $\mathfrak{c} \in (0,\infty)$.
    The numerator is equal to
    \begin{align*}
        \E \Big[e^{2\Re G_{N,\varepsilon}(x)}\Big]
        &=
        \exp\bigg(\sum_{p\leq N}\frac{|f(p)|^2}{p^{1+2\varepsilon}}\bigg)\,,
    \end{align*}
    using the formula for the Laplace transform of a Gaussian. 
    The denominator is evaluated using independence of the $\alpha_p$ and $\E[\alpha_p^k] = 0$ for $k\neq 0$:
    \begin{multline*}
        \E\big[\vert  D_{N}(\alpha f,\tfrac12+\varepsilon+ix)\vert^{2}\big] \\
        =\prod_{p\leq N}\E\Bigg[\bigg|1+\sum_{j\geq1}\frac{\alpha_p^jf(p^j)}{p^{j(\frac12+\ep+ix)}}\bigg|^2\Bigg]  = \prod_{p\le N} \sum_{j,k\ge 0} \, \frac{f(p^k)\overline{f(p^j)}}{p^{\frac{j+k}{2}+(j+k) \varepsilon}} \, p^{-ikx+ijx} \, \mathbb{E}\big[\alpha_p^{k-j}\big]
        = \prod_{p\le N} \sum_{k\ge 0} \frac{|f(p^k)|^2}{p^{k(1+2\varepsilon)}}\,.
    \end{multline*}
    Taking the logarithm of the ratio, it therefore suffices to show convergence as $\min(N,1/\varepsilon)\to\infty$ of
    \begin{equation}\label{eq:DomConvPre}
        \sum_{p\le N} \bigg( \log\bigg(1+\frac{|f(p)|^2}{p^{1+2\varepsilon}}+\sum_{k\ge 2} \frac{|f(p^k)|^2}{p^{k(1+2\varepsilon)}}\bigg)-\frac{|f(p)|^2}{p^{1+2\varepsilon}}\bigg)\,.
    \end{equation}
    We  do this by dominated convergence. Because each summand in the last display converges as $\min(N,1/\varepsilon)\to\infty$, we only need to show  they are dominated by a summable sequence.
    From Taylor's theorem and the inequality $2ab \leq a^2+b^2$, we have $|\log(1+a+b) - a| \ll a^2+b+b^2$ for $a, b\geq 0$ bounded.  Since the twist condition \eqref{eq:TwistCondition} implies the supremum of the summands in \eqref{eq:DomConvPre} is finite, we may apply this inequality to obtain
    \begin{align*}
        \bigg|\log\bigg(1+\frac{|f(p)|^2}{p^{1+2\varepsilon}}+\sum_{k\ge 2}\frac{|f(p^k)|^2}{p^{k(1+2\varepsilon)}}\bigg)-\frac{|f(p)|^2}{p^{1+2\varepsilon}}\bigg|
        \ll \frac{|f(p)|^4}{p^2} + \sum_{k\geq 2} \frac{|f(p^k)|^2}{p^k}\,,
    \end{align*}
    where the implied constant is independent of $p$ and $\ep$. Now,  $\sum_p |f(p)|^4/p^2<\infty$ because $|f|^2\in \mathbf{P}_{\theta}$ implies $\sup |f(p)|/\sqrt{p} <\infty$, while the $|f(p)|^3$ term in \eqref{eq:TwistCondition} implies $\sum_p |f(p)|^3/p^{3/2} <\infty$. The sum over $k\geq 2$ in the right-hand  of the above display is similarly bounded by \eqref{eq:TwistCondition}. Thus, the right-hand side is summable over $p$, implying convergence of \eqref{eq:DomConvPre}. This proves \eqref{eq:proof-thm1-pfratio}.

    We just showed $X_{N,\ep}$ converges uniformly  over $\cI$ in probability to $X(x):= \mathfrak{c} E^{(1)}(x) \exp(2\Re E(x))$, and $X = \d \mu_\infty/\d \mu_\infty^{\text{GMC}}$. We now establish the properties of $X$ stated in \cref{thm:main}. Finiteness of positive moments of $X$ follows from the  moment bounds in \cref{prop:error,prop:off-critical-error}. 
    The zero set of $X$ is equal to the zero set of $E^{(1)}$, which  by \cref{prop:error} is the zero set of $D_{N_0}(\alpha f, \frac12+ix)$. Call this random set $\mathcal{Z}$. 
    
    We now show $\mu_\infty^{\text{GMC}}(\mathcal{Z}) = 0$, almost surely. 
    Recall $N_0$ from the definition of $\mathcal{Z}$ in \cref{prop:error}. We may assume in the coupling of $\alpha_p$ and $Z_p$ constructed in \cref{prop:coupling} that $(\alpha_p)_{p\leq N_0}$ is independent of $(Z_p)_p$, and thus of $\mu_\infty^{\text{GMC}}$: of course, this changes nothing for the conclusion of \cref{prop:coupling} since $N_0<\infty$.
    Next, define $F_p(z) \coloneqq 1+ \sum_{j\geq 1} \frac{f(p^j)}{p^{j/2}} z^j$ and $\alpha_p(x) \coloneqq \alpha_p /p^{ix}$. Observe for each fixed $x\in\cI$, the $(\alpha_p(x))_{p\leq N_0}$ are still iid Steinhaus independent of $(Z_p)_p$, and that 
    \[
    \mathcal{Z} = \bigcup_{p\leq N_0} \{x\in\cI : F_p(\alpha_p(x)) = 0\}\,. 
    \]
    Letting $\mathcal{F}$ denote the $\sigma$-algebra generated by $(Z_p)_p$, we compute via independence and Tonelli:
    \begin{multline*}
        \E\big[\mu_\infty^{\text{GMC}}(\mathcal{Z}) \given \mathcal{F} \big] 
        \leq \sum_{p\leq N_0} \E\big[ \mu_\infty^{\text{GMC}}(\{x\in\cI : F_p(\alpha_p(x)) = 0\}) \given \mathcal{F}\big]  \\
        = \sum_{p\leq N_0} \E\bigg[\int_{\cI} \ind{ F_p(\alpha_p(x)) = 0} \mu_\infty^{\text{GMC}}(\d x)  \given \mathcal{F}\bigg] 
        = \sum_{p\leq N_0} \int_{\cI} \P(F_p(\alpha_p(x)) =0) \, \mu_\infty^{\text{GMC}}(\d x)\,.
    \end{multline*}
    We now show $\P(F_p(\alpha_p(x)) =0)=0$ for all $x \in \cI$. Note that $F_p(z)$ converges absolutely for all $|z|\leq 1$ due to the last term in the twist condition \eqref{eq:TwistCondition}. As such, $F_p$ is analytic in $\mathbb{D}$, continuous on $\overline{\mathbb{D}}$, and not identically $0$ since $F_p(0) = 1$. It follows that the zero set of $F_p|_{\S^1}$ has vanishing Lebesgue measure (see \cite[Theorem~17.18]{rudin1987real} and the discussion below its proof). Since $\alpha_p(x)$ is Steinhaus for each $x$, we find $\P(F_p(\alpha_p(x)) =0)=0$. Thus, by countable additivity, $\mu_\infty^{\text{GMC}}(\{x\in \R : X(x) = 0\}) = 0$ almost surely. Mutual absolute continuity (in particular, $\mu_{\infty}^{\text{GMC}} \ll \mu_\infty$) follows immediately. 

    It remains to show $L^r$ convergence for $\theta \in (0,1)$ and $r\in[1,1/\theta)$. Fix $ h\in C(\cI)$. Choose $r< r' <r'' <1/\theta$, and let  $q''$  denote the conjugate Hölder exponent to $r''/r'$. 
    Recall $\|\cdot\|_\infty$ denotes the $L^{\infty}(\cI)$ norm. Then the triangle inequality, 
    Hölder's inequality, and $\|h\|_\infty <\infty$ imply
    \[
        \sup_{N,\ep\geq 0} \E|\mu_{N,\ep}(h)|^{r'} \leq \|h\|_\infty^{r'} \sup_{N,\ep \geq 0}  \E\Big[ \|X_{N,\ep}\|_\infty^{r'}  \, \mu_{N,\ep}^{\text{GMC}}(\cI)^{r'} \Big]  
        \ll \sup_{N,\ep} \E[\|X_{N,\ep}\|_\infty^{r' \cdot q''}]^{\frac{1}{q''}} \, \E[\mu_{N,\ep}^{\text{GMC}}(\cI)^{r''}]^{\frac{r'}{r''}} \,.
    \]
    The right-hand side is finite due to the above calculations involving $X_{N,\ep}$ and \cref{prop:gmc}. The same calculations show $\E|\mu_{\infty}(h)|^{r'}<\infty$, and thus $|\mu_{N,\ep}(h)|^r$ is uniformly integrable. We have already shown convergence in probability, so Vitali convergence theorem yields $L^r$ convergence.
\end{proof}

\begin{remark}[Rademacher RMFs] \label{rk:rademacher}
We expect our method to also treat the \emph{Rademacher} case, where $\alpha:\N\to \C$ is replaced by the Rademacher RMF $r:\N\to\R$, defined in \cref{subsubsec:numbertheory-motivation}, and the twists $f$ take values in $\R$. In fact, we write our preliminaries in Section~\ref{sec:prelim} such that the proofs of \cref{prop:coupling,prop:off-critical-error} hold upon replacing $\alpha$ with $r$ without \emph{any} modification. The proof of \cref{prop:error} actually simplifies a bit, since $r(n)$ vanishes when $n$ is not square-free and thus the $j\geq 2$ terms vanish in the definition \eqref{eq:TruncatedEP} of $D_N$. 
The main  addition required for \cref{thm:main} to hold is the universality of GMC for $2 \Re G_{N,\ep}$, i.e., the analog of \cref{prop:gmc}, which is less standard and more technical in the Rademacher case because $\Re G_{N,\ep}$ is deterministically an even function, so that its covariance function $C_{N,\ep}(x,y)$ has singularities at both $x=y$ and $x= -y$, and thus a doubled singularity at $x= y = 0$. We leave this  for future work.
\end{remark}

\section{Preliminaries: sums of independent random variables}
\label{sec:prelim}
We start by developing necessary preliminaries relating to Gaussian coupling, subgaussian estimates, and sums of random variables in Hilbert spaces.
Below, we let $(A_p)_p$ denote a sequence of iid random variables satisfying one of Case \caseS or Case \caseR below:
\begin{enumerate}
    \item[(S)] $(A_p)_p$ are either iid Steinhaus or iid $\cN^\C(0,1)$ random variables, where we recall $\cN^\C(0,1)$ denotes the law of $(X+iY)/\sqrt2$, where $X$ and $Y$ are iid standard Gaussians in $\R$. 
    \label{case:S}
    \item[(R)] $(A_p)_p$ are either iid Rademacher or iid standard Gaussians in $\R$.
    \label{case:R}
\end{enumerate}

\subsection{Gaussian coupling}\label{subsec:gaussian-coupling-tools}
We first record a Berry--Esseen theorem for sums of independent random vectors, taken from the book of Bhattacharya--Rao\footnote{The precise reference is Theorem 17.6 of \cite{bhattacharya-rao} specialized  with $s=3$, $X_j\coloneq X_j/\sqrt{n}$,  and $C$ to be the Euclidean ball of radius $r$ centered at $0$.} \cite{bhattacharya-rao} and first proved by Rotar \cite{rotar}.

\begin{lemma}[{\cite[Theorem~17.6]{bhattacharya-rao}}] \label{lem:berry-esseen}
Fix $k, n  \in \N$. Let $(X_i)_{i\in \N}$ denote a sequence of  independent random variables taking values in $\R^k$ such that
\[
  \E X_j = 0 \,,\, 1 \leq j \leq n\qquad \text{and} \qquad \sum_{j=1}^n \Cov(X_j) = \mathrm{Id}\,.
\]
Let $Z$ denote a standard Gaussian in $\R^k$. There exist constants $c,C>0$  such that if 
\[
   \sum_{j=1}^n \E|X_j|^3 \leq c\,,
\]
then for all $n\in \N$ and $r\geq 0$, 
\[
  \bigg|\P\bigg( \bigg|\sum_{j=1}^n X_j\bigg|  \leq r \bigg) - \P (|Z| \leq r )\bigg|\leq \frac{C}{(1+r)^3} \sum_{j=1}^n \E|X_j|^3\,.
\]
\end{lemma}

Before proceeding, recall the Wasserstein-1 distance (a.k.a.\ Kantorovich-Rubinstein metric) between two random variables $X$ and $Y$ taking values in a Polish space $M$ with metric $d$:
\[
  W_1(X,Y) \coloneqq \inf_{\gamma \in \Gamma(X,Y)} \E_{\gamma} [d(X,Y)]\,,
\] 
where $\Gamma(X,Y)$ denotes the set of couplings between the laws of $X$ and $Y$, and the expectation is taken with $(X,Y) \sim \gamma$. 
When $X$ and $Y$ take values in $\R$, we have the formula of Vallender \cite{vallender}:
\begin{align}\label{eq:vallender}
  W_1(X,Y) = \int_{\R} |\P(X \leq r) - \P(Y \leq r)| \d r\,.
\end{align}

For conciseness, for a collection of complex scalars $\{\phi_p\}_p$ and a finite set of primes $B$, we define 
\[
\langle\phi A\rangle_{B}\coloneq \sum_{p \in B} \phi_p A_p \,, \qquad\qquad
  V_{B} \coloneqq  \sum_{p \in B} |\phi_p|^2\,, \qquad \qquad T_{B} \coloneqq \sum_{p\in B} |\phi_p|^3\,.
\]
\begin{lemma}\label{lem:block-coupling}
Consider $(A_p)_p$ in  Case \caseS or Case \caseR.
There exist constants $c, C>0$ such that for any finite set of primes $B$, and for any  $\{\phi_p\}_p \subset \C$ in Case \caseS or $\{\phi_p\}_p \subset \R$ in Case \caseR, there exists a coupling of $\langle\phi A\rangle_B$ with a  Gaussian $Z_{B}$ such that 
\[
  \mathrm{Cov}(Z_B) = \mathrm{Cov}\langle\phi A\rangle_B\,,
\]
and, letting $\Delta_{B} \coloneqq |\langle\phi A\rangle_B - Z_{B}|$, we have the first moment bound
\begin{align} \label{eq:block-coupling-expectation}
  \E \Delta_{B} \leq C \min\bigg( V_{B}^{1/2} \,,\, \frac{T_{B}}{V_{B}} \bigg)
\end{align}
and the tail bound
\begin{align}\label{eq:block-coupling-tail}
  \P(\Delta_{B} > u ) \leq Ce^{-c u^2/V_B}
\end{align}
for all $u\geq 0$, where we use the interpretation that $1/V_B \coloneqq 0$ if $V_B = 0$.

\begin{proof}
Since the case $V_B = 0$ is trivial, we assume throughout the proof $V_B>0$. Moreover, it is quick to show $\E\Delta_B \leq C V_B^{1/2}$: we simply take the independent coupling between $\langle\phi\alpha\rangle_B$ and $Z_B$ (they are realized as independent random variables on the same probability space). Then $\E\Delta_B \leq \E|\langle\phi A\rangle_B| + \E|Z_B| \leq C V_B^{1/2}$ by Jensen's inequality.

To prove \eqref{eq:block-coupling-expectation}, it  remains to show $\mathbb{E}\Delta_B \le CT_B/V_B$. 
The proofs are identical in Case \caseS and Case \caseR, so for concreteness we take $A_p = \alpha_p$ iid Steinhaus.
We  apply \cref{lem:berry-esseen} in dimension $k=2$, identifying $\C$ with $\R^2$.  Observe 
$\Cov \langle\phi\alpha\rangle_B = \tfrac12 V_B \id$. 
Letting $X_p \coloneq \sqrt{\tfrac{2}{V_B}} \phi_p \alpha_p$, we then have 
\[
\sum_{p \in B} \Cov(X_p) = \id \qquad\text{and}\qquad
\sum_{p\in B} \E |X_p|^3 = \big(\tfrac{2}{V_B}\big)^{3/2} \,T_B\,. 
\]
The third-moment condition of \cref{lem:berry-esseen} then becomes, for some constant $c>0$,
\begin{equation}\label{eq:third-mom-condition-proof}
T_B/V_B^{3/2}\leq c
\end{equation}

Suppose first that \eqref{eq:third-mom-condition-proof} is satisfied.  Then \cref{lem:berry-esseen} yields a Gaussian vector $Z\sim \cN(0,\id)$ with
\begin{align}
  \Big|\P\Big( \Big|\sqrt{\tfrac{2}{V_B}} \langle\phi\alpha\rangle_B\Big|  \leq r \Big) - \P (|Z| \leq r )\Big|\ll \frac{1}{(1+r)^3} \frac{T_B}{V_B^{3/2}}
\end{align}
for all $r\geq 0$.
Observe that the $W_1$ distance between two rotationally-symmetric random variables is bounded by the $W_1$ distance between their norms: the same uniform random variable can be used for both angles, taken independently of the norms. 
We may then use the formula \eqref{eq:vallender}
to compute
\begin{align*}
  W_1\Big(\sqrt{\tfrac{2}{V_B}} \langle\phi\alpha\rangle_B, Z\Big) \le W_1\Big(\Big|\sqrt{\tfrac{2}{V_B}} \langle\phi\alpha\rangle_B\Big|, |Z| \Big) = \int_0^{\infty} \Big| \P\Big(\Big|\sqrt{\tfrac{2}{V_B}} \langle\phi\alpha\rangle_B\Big| \leq r\Big) - \P(|Z| \leq r) \Big| \ll \frac{T_B}{V_B^{3/2}}\,.
\end{align*}
Scaling each random variable by $\sqrt{V_B/2}$ yields $W_1(\langle\phi\alpha\rangle_B, Z_B) \ll T_B/V_B$, where $Z_B = \sqrt{V_B/2} Z$ and therefore $\Cov \langle\phi\alpha\rangle_B = \Cov Z_B$. This completes the proof of \eqref{eq:block-coupling-expectation} when \eqref{eq:third-mom-condition-proof} holds.

If \eqref{eq:third-mom-condition-proof} does not hold, i.e.,
$T_B/V_B^{3/2} > c$,
then the independent coupling of $\langle\phi\alpha\rangle_B$ and $Z_B$ yields 
\[
  \E\Big|\sqrt{\tfrac{2}{V_B}}(\langle\phi\alpha\rangle_B - Z_B)\Big| \leq  \E\Big|\sqrt{\tfrac{2}{V_B}}\langle\phi\alpha\rangle_B\Big| + \E |Z| \leq C'
\]
for some constant $C'>0$. Choose $C$ such that $C' < cC < T_B/V_B^{3/2}$. It follows that 
\[
  E\Delta_B \leq C V_B^{1/2} \frac{T_B}{V_B^{3/2}} = C \frac{T_B}{V_B}\,.
\]

The tail bound \eqref{eq:block-coupling-tail} follows from the triangle inequality and Hoeffding's inequality. Choose any coupling such that \eqref{eq:block-coupling-expectation} holds, and write
\[
\P(\Delta_B >u) \leq \P(|\langle\phi\alpha\rangle_B| > u/2) + \P(|Z_B|> u/2)\,.
\]
Fix any finite mesh of angles $\Theta \subset \S^1$ such that $\cos(|\theta- \theta'|_{\S^1}) \geq 1/2$  for adjacent $\theta, \theta' \in \Theta$. Then for any $|x| >u$, there exists $\theta \in \Theta$ such that $\theta \cdot x = |x| \cos |\theta - \theta_x|_{\S^1} \geq |x|/2$, where $\theta_x \coloneqq x/|x|$. A union bound and Hoeffding's inequality  yield
\[
  \P(|\langle\phi\alpha\rangle_B| >u/2) \leq \sum_{\theta \in \Theta}\P(\theta  \cdot\langle\phi\alpha\rangle_B > u/4) \leq C e^{-cu^2/V_B}\,.
\]
The same bound on $\P(|Z_B|>u/2)$ follows from standard Gaussian tail estimates.
\end{proof}
\end{lemma}

\subsection{Subgaussian estimates}
We now record two estimates on subgaussian random variables.
\begin{lemma}\label{lem:small-mean-subgaussian}
Fix constants $C',c', K>0$. 
Let $X$ be a nonnegative random variable such that $\E X \leq r$ and, for some $\sigma \in (0,K]$, 
\[
\P(X>u ) \leq C' e^{- c' u^2/\sigma^2}\,.
\]
Then for each $\lambda>0$, there exist constants $C_\lambda, c_\lambda>0$ (depending also on $c',C'$, and $K$) such that
\[
  \E e^{\lambda X}-1 \leq C_\lambda r + C_\lambda e^{-c_\lambda/\sigma}\,.
\]
\begin{proof}
On the event $\{X \leq \sigma^{1/2}\}$, the mean value theorem yields
\[
  \E[(e^{\lambda X} - 1) \ind{X \leq \sigma^{1/2}}] \leq \E[\lambda e^{\lambda \sigma^{1/2}} X \ind{X \leq \sigma^{1/2}}] \leq C_\lambda r\,,
\]
where the $\sigma$ independence in the last bound comes from $\sigma\leq K$.
On  $\{X > \sigma^{1/2}\}$, we write
\begin{multline}\label{eq:layer-cake}
  \E [(e^{\lambda X}-1) \ind{X > \sigma^{1/2}}] - (e^{\lambda \sigma^{1/2}}-1)\P(X>\sigma^{1/2})
  \\
  = \int_{\sigma^{1/2}}^{\infty} \int_0^x \partial_u (e^{\lambda u}) \d u ~\d \P_X(x)  
  = \int_{\sigma^{1/2}}^{\infty} \P(X>u) \lambda e^{\lambda u} \d  u  \leq C_\lambda \int_{\sigma^{1/2}}^{\infty} e^{-c'\frac{u^2}{\sigma^2} + \lambda u} \d u\,.
\end{multline}
Now, for $u\geq \sigma^{1/2}$ and $\lambda>0$, there exists $\sigma_\lambda>0$ such that for all $\sigma \in (0,\sigma_\lambda]$, we have 
$
  \lambda u \leq \tfrac{c' u^2}{2\sigma^2}$.
For such $\sigma$, it follows from \eqref{eq:layer-cake} that
$
  \E [(e^{\lambda X}-1) \ind{X > \sigma^{1/2}}] \leq C_\lambda e^{-c_\lambda/\sigma}\,.
$
For $\sigma >\sigma_\lambda$, note that the right-hand side of \eqref{eq:layer-cake} can be bounded by a constant $C_\lambda'>0$ times the Laplace transform of a Gaussian with variance $K^2/2c'$
to get 
$\E [(e^{\lambda X}-1) \ind{X > \sigma^{1/2}}] \leq C_\lambda$,
possibly enlarging $C_\lambda$ from before.
We finish by multiplying by $e^{-c_\lambda/\sigma} e^{c_\lambda/\sigma_\lambda}$, which is larger than $1$. 
\end{proof}
\end{lemma}

\begin{lemma}\label{lem:sup-subgaussian-expmom}
Let $(b_n)_{n\in \N}\subset \R$ be a sequence  satisfying $b_n >0$ and $\sum_{n\in\N} b_n <\infty$. Consider a sequence of nonnegative random variables $U_n$ satisfying, for some $c,C>0$ and for all $u\geq 0$,
\[
  \P(U_n>u) \leq C b_n^{-2} e^{-cu^2/b_n^2}\,.
\]
Then $U\coloneqq \sup_{n\in\N} U_n$ satisfies $\E e^{\lambda U}<\infty$ for all $\lambda>0$.
\begin{proof}
We compute the tail of $U$ via union bound:
\[
  \P(U > u) \leq C \sum_{n\in\N} b_n^{-2} e^{-c u^2/b_n^2}\,.
\]
Observe that  $v^{3} e^{-cv^2/2} \leq C'$ for some constant $C'>0$ and all $v \geq 0$, and thus for all $u, n\geq 1$,
\[
  b_n^{-2} e^{-\frac{cu^2}{2b_n^2}}\leq C' u^{-3} b_n\,.
\]
Since $b_n \to 0$, we have $\eta \coloneqq \sup_{n\in\N} b_n <\infty$. It follows that, for $u\geq 1$,
\[
    \P(U>u) \ll u^{-3} e^{-\frac{cu^2}{2\eta^2}} \sum_{n\in \N} b_n \ll u^{-3} e^{-\frac{cu^2}{2\eta^2}} \,.
\]
This implies the exponential moment bound for any $\lambda >0$ via the layer-cake formula.
\end{proof}
\end{lemma}

\subsection{Sums of random variables in Hilbert spaces}\label{subsec:hilbert}
Let $\cI \subset \R$ be a fixed compact interval throughout this section.
We are interested in series of the form $F_B(x) \coloneqq \sum_{p \in B} A_p f_p(x)$, where the $(A_p)_p$ satisfy Case \caseS or \caseR and each $f_p \in C^{1}(\mathcal{I})$. In particular, we will need various bounds on $\|F_B\|_\infty$. The idea, already appearing in \cite{SW20}, is to apply the Sobolev embedding $\|g\|_\infty \leq C_{\cI} \|g\|_{H^1} \coloneqq C_{\cI}(\|g\|^2_2 + \|g'\|^2_2)^{1/2}$, where $H^1$ denotes $H^1(\cI;\C)$ or $H^1(\cI;\R)$ and $C_{\cI}>0$ is a constant.
The advantage of working in a Hilbert space is that the $A_p$ are independent with  $\E|A_p|^2 =1$, and thus
\[
  \E\|F_B\|_\infty^2 \ll \E\|F_B\|_{H^1}^2 =  \sum_{p \in B} \|f_p\|_{H^1}^2 \,.
\]
With this in mind, we record the main result of this section.

\begin{lemma}\label{lem:series-summability}
Let $\mathcal{B}$ be a collection of disjoint finite subsets of primes. 
Let $H^1$ denote $H^1(\cI; \mathbb{K})$ with $\mathbb{K} = \R$ or $\C$. Fix  $(f_p)_{p}\subset H^1$. If $\mathbb{K} = \R$, let  $(A_p)_{p}$ satisfy Case \caseR, and if $\mathbb{K} = \C$, let $(A_p)_p$ satisfy Case \caseS.
For each $B \in \mathcal{B}$, fix an enumeration $B = \{p_1, \dots, p_{|B|}\}$. 
If the block summability~condition  
\begin{align}\label{eq:block-summability-condition}
    \sum_{B \in \mathcal{B}} \Big(\sum_{p\in B} \|f_p\|_{H^1}^2 \Big)^{1/2}<\infty
\end{align}
holds, then for every $\lambda >0$,
\[
    \E\exp\bigg(\lambda \sum_{B \in \mathcal{B}} \max_{1\leq k\leq|B|} \Big\|\sum_{j=1}^k A_{p_j} f_{p_j} \Big\|_\infty \bigg)<\infty \,.
\]
In particular,  $\sum_{B\in\mathcal{B}} \sum_{p\in B} A_p f_p$ converges uniformly over $\cI$, almost surely. 
\end{lemma}

\cref{lem:series-summability}  can be compared with \cite[Lemma~3.1]{SW20}: under the condition \eqref{eq:block-summability-condition}, we upgrade the convergence to ``blockwise''  absolute convergence. This is crucial for our arguments.
We prove the lemma by developing quantitative estimates on finite sums. We begin with
 the Kahane-Khintchine inequality, a classical result for sums of elements in Banach spaces with iid coefficients~\cite{KahaneBook,LedouxTalagrand}.

\begin{lemma}[Subgaussian Kahane--Khintchine inequality]
\label{lem:subgaussian-kahane-khintchine}
Let $H$ be a real or complex Hilbert space and fix $(f_i)_{i\in\N}\subset H$.
If $H$ is a real Hilbert space, let $(A_i)_{i\in\N}$ satisfy Case \caseR, and if $H$ is a complex Hilbert space, let $(A_i)_{i\in\N}$ satisfy \caseS.
For any $m\in\N$, 
define  
\[
S_m\coloneqq\sum_{k=1}^m A_k f_k \qquad \text{and} \qquad \sigma_m^2\coloneqq\E\|S_m\|_H^2 =  \sum_{k=1}^m \|f_k\|_H^2 \,.
\]
There exist $c,C>0$ such that for all $m \in \N$ such that  $\sigma_m >0$,
and for all $u\geq 0$,
\begin{align}\label{eq:subgaussian-kahane-khintchine-tail}
        \P(\|S_m\|_H>u)
        \le
        Ce^{-cu^2/\sigma_m^2}\,.
\end{align}
\begin{proof}
A version of the Kahane--Khintchine inequality stated explicitly for Rademacher, Steinhaus, and Gaussian coefficients can be found in \cite[Proposition~6.2.7]{kahane-khintchine-source}\footnote{They call Steinhaus random variables ``complex Rademacher'' random variables.}, which states the bound $\E e^{\delta\|S_m\|_H^2} \leq (1- 2e \delta \E\|S_m\|_H^2)^{-1}$ for all $\delta \in (0, (2e\sigma_m^2)^{-1})$. Equation~\eqref{eq:subgaussian-kahane-khintchine-tail} now follows from taking $\delta = 1/4e\sigma_m^2$ and the Markov inequality.
\end{proof}
\end{lemma}

\begin{corollary}
\label{cor:finite-block-H1}
Follow the same setup and notation as \Cref{lem:subgaussian-kahane-khintchine} with $H \coloneqq H^1(\mathcal{I};\mathbb{K})$ and $\mathbb{K} = \R$ or $\C$. 
There exist $c,C>0$ such that for all $m\in\N$ and $u\geq0$, we have
\begin{align}\label{eq:finite-block-maxtail}
        \P\Big(
        \max_{0\le k\le m}\|S_k\|_\infty>u
        \Big)
        \le
        Ce^{-cu^2/\sigma_m^2}\,.
\end{align}
where we use the convention that $1/\sigma_m^2 \coloneqq 0$ when $\sigma_m = 0$.
Moreover, for all $\lambda,K>0$, there exists $C_\lambda<\infty$ (also depending on $K$) such
that whenever $\sigma_m \in [0,K]$, we have
\begin{align}\label{eq:finite-block-expmoment}
        \E\big[\exp( \lambda \max_{0\leq k\leq m} \|S_k\|_\infty)\big]-1
        \leq
        C_\lambda\sigma_m .
\end{align}
\end{corollary}

\begin{proof}
Assume $\sigma_m>0$. The bound \eqref{eq:finite-block-maxtail} follows from \eqref{eq:subgaussian-kahane-khintchine-tail} by the Sobolev embedding $\|S_k\|_\infty \leq 2\|S_k\|_{H^1}$ and from L\'evy's maximal inequality \cite[Lemma~2.3.1]{KahaneBook}. 

We now prove \eqref{eq:finite-block-expmoment} using  \cref{lem:small-mean-subgaussian} with $X \coloneqq \max_{0\leq k \leq m} \|S_k\|_\infty$. 
Integrating the tail bound \eqref{eq:finite-block-maxtail} yields $\E X \leq C \sigma$. 
We may then apply \Cref{lem:small-mean-subgaussian} with
$r= C \sigma$, to get
\[
        \E e^{\lambda X}-1
        \leq
        C_\lambda\sigma
        +
        C_\lambda e^{-c_\lambda/\sigma}.
\]
Since $\sigma \in (0,K]$,  we have $e^{-c_\lambda/\sigma}\le C_\lambda\sigma$ upon possibly enlarging $C_\lambda$. This concludes the proof.
\end{proof}

\begin{proof}[Proof of~\cref{lem:series-summability}]
For each $B \in \mathcal{B}$, define
\[
    S_B \coloneqq \sum_{p\in B} A_p f_p \,, \qquad \sigma_B^2  \coloneqq \E \|S_B\|_{H^1}^2 = \sum_{p\in B}  \|f_p\|_{H^1}^2\,.
\]
The block summability condition \cref{eq:block-summability-condition} $\sum_{B\in \cB} \sigma_B <\infty$ implies that $K\coloneqq \sup_{B \in \cB} \sigma_B$ is finite. Thus, we may apply \eqref{eq:finite-block-expmoment}. This along with 
the independence of the $A_p$ yields 
\[
    \E\exp\bigg(\lambda \sum_{B \in \mathcal{B}} \max_{1\leq k\leq|B|} \Big\|\sum_{j=1}^k A_{p_j} f_{p_j} \Big\|_\infty \bigg)
    \leq  \prod_{B\in\cB} (1+ C_{\lambda} \sigma_B ) \leq \exp\bigg(C_\lambda \sum_{B\in\cB} \sigma_B\bigg) <\infty\,.
\]
This proves the lemma.
\end{proof}

\section{Gaussian approximation on the critical line: proof of \texorpdfstring{\cref{prop:coupling}}{Proposition 2.2}}\label{sec:gaussian-approximation-critical-line}
This section is dedicated to proving \cref{prop:coupling}. Recall the notation there. Throughout this section, we fix a compact interval $\cI \subset \R$. Recall from \cref{subsec:notation} that  we write 
$\|\cdot\|_{\infty} \coloneqq \|\cdot\|_{L^{\infty}(\cI)}$.
Recall also our twist condition \eqref{eq:TwistCondition}, which implies the following:
\begin{align}\label{eq:twist-condition-cube-term}
  \sum_{n =1}^{\infty} \bigg(\sum_{e^n< p \leq e^{n+1}} \frac{|f(p)|^3}{p^{3/2}} \bigg)^{1/3} < \infty\,.
\end{align}
This is the only part of \eqref{eq:TwistCondition} needed for the proof of \cref{prop:coupling}.

This section is organized as follows. 
In \cref{subsec:block-coupling-scheme-proof-thm}, we define our block coupling scheme. We then prove \cref{prop:coupling} conditional on three lemmas. In \cref{subsec:proof-of-lemmas}, we prove these lemmas.

\subsection{Block coupling scheme and proof of \texorpdfstring{\cref{prop:coupling}}{Proposition 2.2}}
\label{subsec:block-coupling-scheme-proof-thm}
We now introduce our block coupling scheme, followed by an overview of the proof of \cref{prop:coupling}.
We begin by introducing the \emph{scales} $I_n \coloneqq \{p \text{ prime} : p \in (e^n, e^{n+1}]\}$, $n\in \N$. It is also  convenient to introduce
  \[
  \phi_p \coloneq \frac{f(p)}{\sqrt{p}}\,, \qquad \qquad
    V_n \coloneqq \sum_{p \in I_n} |\phi_p|^2\,, \qquad \qquad T_n \coloneqq \sum_{p\in I_n} |\phi_p|^3\,.
  \]
For each scale $n$, we introduce the \emph{variance cutoff}
  \begin{align}\label{def:cutn}
    \cut_n \coloneqq T_n^{2/3}\,.
  \end{align}
Observe that our twist condition \eqref{eq:twist-condition-cube-term} implies
\begin{align}\label{eq:twist-condition-cutn}
    \sum_{n\in\N} \cut_n^{1/2} <\infty\,.
\end{align}
We then create ``blocks'' $B \subset I_n$ within each scale $n\in \N$ according to the following algorithm.
  \begin{enumerate}
    \item If $\cut_n = V_n = 0$, we move on to scale $n+1$. In light of this, we take the convention for the rest of this section that $\cut_n^{-1} = 0$ when $\cut_n = 0$.

    \item We enumerate the primes in $I_n$ in increasing order and partition them greedily into blocks $B$, where each block $B \subset I_n$ is completed once its \emph{block variance} 
    \[
        V_B \coloneqq \mathrm{Tr} \Cov \Big(\sum_{p\in B} \phi_p \alpha_p  \Big) = \sum_{p\in B} |\phi_p|^2
    \]
    first reaches or surpasses $\cut_n$.
    Note that for each $p\in I_n$, we have 
    \[
    |\phi_p|^2 = (|\phi_p|^3)^{2/3}\leq T_n^{2/3} = \big(\sum_{p\in I_n} |\phi_p|^3\big)^{2/3}\le \sum_{p\in I_n} |\phi_p|^2 = V_n \,,
    \]
    which implies $|\phi_p|^2\le \mathrm{cut}_n$ for all $p \in I_n$.     Therefore, each completed block $B\subset I_n$ satisfies 
    \begin{equation}\label{eq:complete-light-block-variance-bound}
      \cut_n \leq V_B \leq 2 \cut_n \,.
    \end{equation}
    There will be at most one \emph{incomplete} block $B$ in $I_n$, satisfying 
    \begin{equation} \label{eq:incomplete-light-block-variance-bound}
      V_B < \cut_n\,.
    \end{equation}
  \end{enumerate}
For all blocks $B$, \cref{lem:block-coupling} gives a coupling of $\langle\phi\alpha\rangle_B = \sum_{p\in B} \phi_p\alpha_p$ with a complex Gaussian $Z_B$ of the same covariance as $\langle\phi\alpha\rangle_B$. It follows that we may construct an iid sequence $(Z_p)_{p \text{ prime}} \sim \mathcal{N}^{\C}(0,1)$ on the same probability space such that 
\[
  Z_B = \sum_{p\in B} \phi_p Z_p\,,
\]
since $\Cov \alpha_p = \Cov Z_p$. This constructs the $(Z_p)_{p \text{ prime}}$ appearing in \cref{prop:coupling}. 

Recall we seek to bound $E_{N,0}(x) \coloneqq S_{N,0}(x) - G_{N,0}(x)$. 
Towards this, we define the quantities 
\begin{alignat*}{2}
  p_{\min}(B) &\coloneqq \min \{p \in B\}\,,&\qquad\qquad\qquad  p_{\max}(B) &\coloneqq \max \{p\in B\}\,, \\
  S_B(x) &\coloneqq \sum_{p \in B} \phi_p \alpha_p p^{-ix}\,,& 
  G_B(x) &\coloneqq \sum_{p \in B} \phi_p Z_p p^{-ix}\,, \\
  S_B^{\fro}(x) &\coloneqq p_{\min}(B)^{-ix} \, \langle \phi \alpha \rangle_B\,,& 
  G_B^{\fro}(x) &\coloneqq  p_{\min}(B)^{-ix} \, Z_B \,.
\end{alignat*}
Note that for each $N\in \N$, there is at most one block $B$ not fully contained in $[0,N]$, and therefore not all of its primes are included in the sums defining $S_{N,0}$ and $G_{N,0}$. We refer to this block as the \emph{residual block}. This small nuisance is handled separately by the following lemma.

\begin{lemma}\label{lem:residual-block}
Given a block of primes $B$, enumerate $B = \{p_1< \dots < p_{|B|}\}$. Define  the quantities 
\begin{align*}
  M_B^\alpha &\coloneqq \max_{1 \leq k \leq |B|} \bigg\|\sum_{j=1}^k \phi_{p_j} \alpha_{p_j} p_j^{-ix} \bigg\|_\infty\,,  \qquad M^\alpha \coloneqq \sup_{n\in\N} \max_{B \subset I_n} M_B^\alpha\,,
\end{align*}
and similarly define $M_B^Z$ and $M^Z$ by replacing $\alpha_{p_j}$ with $Z_{p_j}$ above. Then, for all $\lambda >0$, 
\begin{align}\label{eq:residual-block-expmoment}
  \E\big[e^{\lambda M^\alpha}\big] < \infty\,, \qquad \E\big[e^{\lambda M^Z}\big] <\infty \,,
\end{align}
and 
\begin{align}\label{eq:residual-block-0limit}
  \lim_{n\to\infty} \max_{B \subset I_n} M_B^\alpha = \lim_{n\to\infty} \max_{B \subset I_n} M_B^Z = 0\,.
\end{align}
\end{lemma}

Ignoring this residual term, our error function $E_{N,0}$ is then determined by 
the \emph{error from coupling}
\[
    \|S_B^\fro -G_B^\fro\|_\infty \,, \qquad p_{\max}(B) \leq N
\]
and the
\emph{error from freezing}
\[
    \|S_B - S_B^{\fro}\|_{\infty} + \|G_B - G_B^\fro\|_\infty\,,  \qquad p_{\max}(B) \leq N\,.
\]
We now give a short heuristic argument that both the coupling and the freezing error are summable over $B$. For these  heuristics, we ignore the existence of incomplete blocks. 
Observe that the error from coupling is equal to $|\langle\phi \alpha\rangle_B-Z_B|$, which, according to \cref{lem:block-coupling}, has expected size $O(T_B/V_B)$. 
Summing over $B \subset I_n$ and using $V_B \geq \cut_n$ for complete blocks, we find
\[
    \text{total error from coupling on scale $n$}  \approx \sum_{B\subset I_n} \frac{T_B}{V_B} \leq T_n/\cut_n\,.
\]
On the other hand, using the theory of sums of elements in Hilbert spaces with iid coefficients developed in \cref{subsec:hilbert}, we  bound the error from freezing on $B$ via \cref{cor:finite-block-H1} as
\begin{align*}
    \Big(\sum_{p\in B} \|\phi_p \big(p^{-ix} - p_{\min(B)}^{-ix}\big)\Big\|_{H^1}^2 \Big)^{1/2}  \ll V_B^{1/2} (\log p_{\max}(B) - \log p_{\min(B)})\,.
\end{align*}
Observe that the blocks $B \subset I_n$ being disjoint yields   $\sum_{B} (\log p_{\max(B)} - \log p_{\min(B)}) \leq \log e^{n+1} - \log e^n  = 1$. This along with  $V_B \leq 2 \cut_n$ for $B \subset I_n$  yields
\[
    \text{total error from freezing on scale $n$} \approx  \sum_{B\subset I_n} V_B^{1/2} (\log p_{\max}(B) - \log p_{\min(B)}) \ll \cut_n^{1/2}\,.
\]
Balancing these two errors gives a derivation of our choice $\cut_n\coloneqq T_n^{2/3}$, and \eqref{eq:twist-condition-cutn} leads to the above two errors being summable. 

We now formally state the lemmas that bound the errors from coupling and freezing. Then, we  prove \cref{prop:coupling}. 
\begin{lemma}[Error from coupling]\label{lem:frozen-error}
For any $\lambda >0$, 
\[
  \E\big[e^{\lambda \sum_B \|S_B^\fro -G_B^\fro\|_\infty} \big]<\infty\,.
\]
\end{lemma}

\begin{lemma}[Error from freezing]\label{lem:error-from-freezing}
For any $\lambda >0$, 
\[
  \E\big[e^{\lambda \sum_B \|S_B - S_B^{\fro}\|_\infty} \big] <\infty\,, \qquad \E\big[e^{\lambda \sum_B \|G_B- G_B^{\fro}\|_\infty} \big] <\infty\,.
\]
\end{lemma}

\begin{proof}[Proof of \cref{prop:coupling}]
Observe we can write
\[
  E_{N,0}(x) = R_N(x) + \sum_{B : p_{\max}(B) \leq N} (S_B(x) - S_B^\fro(x)) + (S_B^\fro(x) - G_B^\fro(x)) + (G_B^\fro(x) - G_B(x))\,,
\]
where $R_N$ denotes the contribution from the residual block. \cref{eq:couplingthm-exp-moment-N} then follows from H\"older's inequality and Lemmas~\ref{lem:residual-block}--\ref{lem:error-from-freezing}. 

To show the almost-sure convergence of $E_{N,0}$, observe that \cref{lem:error-from-freezing,lem:frozen-error} imply the  series
\[
E(x) \coloneqq \sum_{B} (S_B(x) - S_B^\fro(x)) + (S_B^\fro(x) - G_B^\fro(x)) + (G_B^\fro(x) - G_B(x))
\]
converges absolutely over $B$ and uniformly, almost surely. \cref{lem:residual-block} implies $\|R_N\|_\infty \to 0$ almost surely, and therefore $\|E_{N,0} - E\|_\infty \to 0$ almost surely. 
\end{proof}

\subsection{Proofs of Lemmas~\ref{lem:residual-block}--\ref{lem:error-from-freezing}}
\label{subsec:proof-of-lemmas}
We start by recording a bound on $V_n \coloneq \sum_{p \in I_n} |\phi_p|^2$ that allows us to bound the number of blocks in a given scale $I_n$.
\begin{lemma}\label{lem:number-light-blocks}
There exists a constant $C>0$ such that for all $n\in\N$, we have 
\begin{align}\label{eq:Vn-C/n-bound}
  V_n \leq C/n\,,
\end{align}
and the number of blocks in $I_n$ is bounded by $C \cdot \cut_n^{-1}$.
\begin{proof}
Let $F(x) \coloneqq \sum_{p\leq x} |f(p)|^2$.
Applying Riemann--Stieltjes integration by parts, we get:
\begin{align*}
    V_n &= \int_{(e^n,e^{n+1}]} \frac{1}{x} \d F(x)
    = e^{-(n+1)} F(e^{n+1}) - e^{-n} F(e^n) + \int_{e^n}^{e^{n+1}} \frac{F(x)}{x^2} \d x\,.
\end{align*}
Now, since $|f|^2 \in P_{\theta}$ (see \eqref{eq:PNT}), we have $|F(x)| \ll x/\log x$. The triangle inequality then implies
\[
    V_n \ll \frac{1}{n+1} + \frac{1}{n} + \int_{e^n}^{e^{n+1}} \frac{1}{x \log x}\d x \ll \frac{1}{n}\,.
\]
We can then bound the number of blocks in $I_n$ using this and \eqref{eq:complete-light-block-variance-bound}:
\[
1+ \frac{V_n}{\cut_n} \ll 1+ \frac{1}{n \cdot \cut_n} \ll  \cut_n^{-1}
\]
where the $+1$ comes from the possible existence of an incomplete block in $I_n$.
\end{proof}
\end{lemma}

\begin{proof}[Proof of \cref{lem:residual-block}]
We just prove the results for $M^\alpha$ since the proofs for $M^Z$ are identical. The idea is to use \eqref{eq:finite-block-maxtail} of \cref{cor:finite-block-H1} with $m=|B|$. Write 
\[
  M_{k}^\alpha \coloneqq \sum_{j=1}^k \phi_{p_j} \alpha_{p_j} e^{-ix \log(p_j/p_1)} \,, \qquad M_B^\alpha = \max_{1\leq k \leq |B|} \bigg\|\sum_{j=1}^k \phi_{p_j} \alpha_{p_j} e^{-ix \log(p_j/p_1)} \bigg\|_\infty\,.
\]
Since $B \subseteq I_n$, observe $\log(p_j/p_1) \leq 1$ for all $j \in[1,|B|]$. It follows that 
\[  
  \|e^{-ix \log(p_j/p_1)}\|_{H^1} \ll 1\,,
\]
and thus
\[
  \sigma^2\coloneqq \E\|M_{|B|}^\alpha\|_{H^1}^2 = \sum_{k=1}^{|B|} |\phi_{p_k}|^2\ \|e^{-ix \log(p_k/p_1)}\|_{H^1}^2 \ll \sum_{k=1}^{|B|} |\phi_{p_k}|^2 \ll V_B\,.
\]
Taking $S_k \coloneqq M_k^\alpha$ in \eqref{eq:finite-block-maxtail}, and recalling from \eqref{eq:complete-light-block-variance-bound} that $V_B \leq 2\cut_n$ for $B\subseteq I_n$, we find 
\[
  \P(M_B^\alpha>u) \ll  e^{-c_1 u^2/V_B} \ll e^{-c_2 u^2/ \cut_n} \,.
\]
Using the bound on the number of blocks in $I_n$ from \cref{lem:number-light-blocks} and a union bound, we have
\begin{align}\label{eq:residual-proof-tail}
  \P\Big(\max_{B\subset I_n} M_B^\alpha > u \Big) \ll \cut_n^{-1} e^{-c_3u^2/\cut_n}\,.
\end{align}
 \cref{lem:sup-subgaussian-expmom}  implies $\E[\exp(\lambda M^\alpha)] <\infty$ for all $\lambda \geq0$, proving \eqref{eq:residual-block-expmoment}.
Since $ \cut_n= T_n^{2/3}$ converges to 0 as $n\to\infty$, \eqref{eq:residual-block-0limit} follows from  Borel-Cantelli and taking $u = \delta$ for any $\delta>0$ in~\eqref{eq:residual-proof-tail}.
\end{proof}

\begin{proof}[Proof of \cref{lem:frozen-error}]
First, we have
\[
    \big\| S_B^{\fro} - G_B^{\fro} \big\|_\infty  = \Big | \sum_{p\in B} \phi_p (\alpha_p - Z_p) \Big| = \Delta_B\,,
\]
using the notation of \cref{lem:block-coupling}. 
Using the independence of the blocks, it  suffices to show
\begin{align}\label{eq:frozen-error-sufficient}
    \sum_{n\in\N} \sum_{B \subset I_n} \E\big[e^{\lambda \Delta_B} -1 \big] <\infty\,.
\end{align}
Now, \cref{lem:small-mean-subgaussian} with $\sigma \coloneq V_B^{1/2}$ and $K\coloneq \max_B V_B^{1/2}$ (which is finite because, for $B\subset I_n$, $V_B \leq V_n \to 0$ as $n\to\infty$ by \cref{eq:Vn-C/n-bound}) yields 
\begin{align}\label{eq:frozen-error-1}
    \E e^{\lambda \Delta_B} -1 \leq C_\lambda \E \Delta_B + C_\lambda e^{-c_\lambda/V_B^{1/2}}\,.
\end{align}
Consider a block $B$ in $I_n$.  
If $B$ is incomplete, we have $V_B < \cut_n$ from \eqref{eq:incomplete-light-block-variance-bound}, and thus \cref{lem:block-coupling} gives 
$\E\Delta_B \ll  V_B^{1/2} \ll \cut_n^{1/2}$. If $B$ is complete, we have $V_B \geq \cut_n$ from \eqref{eq:complete-light-block-variance-bound}, and thus \cref{lem:block-coupling} gives $\E\Delta_B \ll T_B/V_B \ll T_B/\cut_n$. Since there is at most one incomplete block per scale $I_n$, we get
\[
\sum_{n\in\N} \sum_{B\subset I_n } \E\Delta_B  \ll \sum_{n\in\N}\bigg( \cut_n^{1/2} + \sum_{\substack{B\subset I_n\\ B \text{ complete}}} \frac{T_B}{V_B} \bigg) \ll \sum_{n\in\N}\bigg(  \cut_n^{1/2} +  \frac{T_n}{\cut_n} \bigg)\ll \sum_{n\in\N} \cut_n^{1/2} <\infty\,,
\]
where in the last bound we used \eqref{eq:twist-condition-cutn}.
To treat the sum over the other term on the right-hand side of \eqref{eq:frozen-error-1}, we use \cref{lem:number-light-blocks}  as well as $V_B \leq 2 \cut_n$ (by \eqref{eq:complete-light-block-variance-bound} and \eqref{eq:incomplete-light-block-variance-bound}) to get 
\[
    \sum_{n\in\N} \sum_{B\subset I_n} e^{-c_\lambda/V_B^{1/2}} \ll  \sum_{n\in\N}  \cut_n^{-1} e^{-c_\lambda /\cut_n^{1/2}} \ll \sum_{n\in\N} \cut_n^{1/2} <\infty\,.
\]
Substituting the above two displays into \eqref{eq:frozen-error-1} yields \eqref{eq:frozen-error-sufficient}, completing the proof. \end{proof}

\begin{proof}[Proof of \cref{lem:error-from-freezing}]
We only prove the result for $S$, as the proof for $G$ is identical. 
Define
\[
    g_p(x) \coloneqq e^{-ix (\log p - \log p_{\min}(B))} -1
\]
for each $p \in B$. Observe that
\[
    S_B(x) - S_B^{\fro}(x) = \sum_{p\in B} \phi_p \alpha_p \Big( e^{-ix \log p} - e^{-ix\log {p_{\min}(B)}}\Big) = e^{-ix \log p_{\min}(B)} \sum_{p\in B} \phi_p \alpha_p g_p(x) \,,
\]
so that $\|S_B -S_B^\fro\|_\infty = \|\sum_{p\in B} \phi_p \alpha_p g_p\|_\infty$. Therefore, taking $A_p \coloneqq \alpha_p$ and $f_p \coloneqq \phi_p g_p$ in \cref{lem:series-summability}, we see that \cref{lem:error-from-freezing} follows once we verify the block summability condition \eqref{eq:block-summability-condition}. Towards this, observe
\[
    \|f_p\|_{H^1}^2 \ll  ( \log p - \log p_{\min}(B))^2|\phi_p|^2 \leq ( \log p_{\max}(B) - \log p_{\min}(B))^2|\phi_p|^2\,.
\]
We then compute
\begin{multline*}
    \sum_{n\in\N} \sum_{B \subset I_n} \Big(\sum_{p \in B} \|f_p\|_{H^1}^2\Big)^{1/2}  \\
    \ll \sum_{n\in\N} \sum_{B\subset I_n} ( \log p_{\max}(B) - \log p_{\min}(B)) V_B^{1/2}  \ll \sum_{n\in\N} \cut_n^{1/2}  \sum_{B\subset I_n} ( \log p_{\max}(B) - \log p_{\min}(B))\,,
\end{multline*}
where the last inequality follows from $V_B\leq 2\cut_n$ (\eqref{eq:complete-light-block-variance-bound} and \eqref{eq:incomplete-light-block-variance-bound}). Now, since  blocks are disjoint, 
\[
\sum_{B \subset I_n } (\log p_{\max}(B) - \log p_{\min}(B)) \leq \log e^{n+1} - \log e^{n} = 1\,.
\]
The above two displays along with the assumption $\sum_{n\in\N} \cut_n^{1/2}<\infty$ together verify  \eqref{eq:block-summability-condition}. 
\end{proof}

\section{Controlling the off-critical error: proof of \texorpdfstring{\cref{prop:off-critical-error}}{Proposition 2.3}} \label{sec:off-critical-errors}
We now give the proof of \cref{prop:off-critical-error}.
\begin{proof}[Proof of \cref{prop:off-critical-error}]
Below, we consider $N \in \N$ and $\ep \geq 0$. Define 
$\index(N) \coloneqq \max \{k \in \N : p_k \leq N\}$.
We then write for any $N\in\N$ and $\ep \geq 0$ the following:
\[
    E_{N,\ep}(x) \coloneqq \sum_{m=1}^{\index(N)} \phi_{p_m} (\alpha_{p_m} - Z_{p_m}) p_m^{-ix} \cdot p_m^{-\ep}\,.
\]
Applying partial summation to  the above display yields:
\begin{align}\label{eq:error-partialsummation}
    E_{N,\ep}(x) &= p_{\index(N)}^{-\ep} \sum_{m=1}^{\index(N)} \phi_{p_m} (\alpha_{p_m} - Z_{p_m}) p_m^{-ix} + \sum_{m=1}^{\index(N)-1} (p_m^{-\ep} - p_{m+1}^{-\ep}) \left( \sum_{j=1}^m \phi_{p_j} (\alpha_{p_j} - Z_{p_j}) {p_j}^{-ix} \right) \nonumber \\
    &= p_{\index(N)}^{-\ep}  E_{N,0}(x)  + \sum_{m=1}^{\index(N)-1} (p_m^{-\ep} - p_{m+1}^{-\ep}) E_{p_m,0}(x)\,.
\end{align}
Observe that the coefficients of the error functions in the above display satisfy the following identity:
\begin{align}\label{eq:summable-coefficients-fixeseverything}
    p_{\index(N)}^{-\ep} + \sum_{m=1}^{\index(N)-1} (p_m^{-\ep} - p_{m+1}^{-\ep})  = p_1^{-\ep}\,.
\end{align}
Recalling that $E$ denotes the almost-sure limit of $E_{N,0}$ from \cref{prop:coupling}, we compute
\begin{align*}
    E_{N,\ep}(x) - E(x) &= p_{\index(N)}^{-\ep}  \big(E_{N,0}(x) - E(x) \big)- \big(1-p_1^{-\ep}\big)E(x) \\
    &\qquad+ \sum_{m=1}^{\index(N)-1} (p_m^{-\ep} - p_{m+1}^{-\ep}) \cdot  (E_{p_m,0}(x)-E(x)) \\
    &= o(1) + \sum_{m=1}^{\index(N)-1} (p_m^{-\ep} - p_{m+1}^{-\ep}) \cdot  (E_{p_m,0}(x)-E(x))\,,
\end{align*}
where the $o(1)$  tends to $0$ uniformly over $x$ almost surely as $\min(N, 1/\ep)\to\infty$.
Fix $\delta >0$. By \cref{prop:coupling}, for almost every $\omega$ in the probability space $\Omega$, there exists $ K(\delta, \omega)<\infty$ such that 
\[
    \sup_{m\geq K} \|E_{p_m,0}-E\|_\infty < \delta\,.
\]
It follows that for every $\delta>0$ and for almost every $\omega$, 
\begin{align*}
    \|E_{N,\ep} - E\|_\infty &\leq o(1) + \bigg\|\sum_{m=1}^{K-1} (p_m^{-\ep} - p_{m+1}^{-\ep})(E_{p_m,0}-E)\bigg\|_\infty + \bigg\| \sum_{m=K}^{\index(N) -1} (p_m^{-\ep} - p_{m+1}^{-\ep}) (E_{p_m,0}-E)\bigg\|_\infty \\
    &\leq o(1)  +  \delta \sum_{m=K}^{\index(N) -1}(p_m^{-\ep} - p_{m+1}^{-\ep})  \\
    &\leq o(1) + \delta\,,
\end{align*}
where in the second line, we used the fact that the sum from $m=1$ to $m=K-1$ has no $N$ dependence and vanishes as $\ep \to 0$. This yields 
\[
    \lim_{\min(N, 1/\ep) \to\infty} E_{N,\ep}(x) = E(x) \qquad \text{uniformly in $x \in \cI$, almost surely.}
\]

Finiteness of the exponential moments of $E_{N,\ep}$ follows from the simple observation that the triangle inequality, \eqref{eq:error-partialsummation}, and \eqref{eq:summable-coefficients-fixeseverything} together imply that 
$\|E_{N,\ep}\|_\infty $ is bounded by a weighted average of $\{\|E_{p_m,0}\|_\infty\}_{m=1}^{\index(N)}$, where the sum of the  weights is bounded by $p_1^{-\ep} \leq 1$. In particular, we have the following for any $N\in\N \cup \{\infty\}$ and $\ep\geq 0$:
\begin{align}\label{eq:partial-summation-trick}
    \|E_{N,\ep}\|_\infty \leq \max_{1 \leq m \leq \index(N)} \|E_{p_m,0}\|_\infty \leq \sup_{N\in \N} \|E_{N,0}\|_\infty\,.
\end{align}
Taking suprema over $N,\ep\geq 0$  in any order yields 
\[
\E\big[e^{\lambda \sup_{N \in \N, \ep \geq 0} \|E_{N,\ep}\|_\infty} \big] \leq \E \big[ e^{\lambda \sup_{N\in \N} \|E_{N,0}\|_\infty} \big]<\infty\,,
\]
where the final relation follows from \eqref{eq:couplingthm-exp-moment-N}. This completes the proof of \cref{prop:off-critical-error}. 
\end{proof}

\section{Proof of \texorpdfstring{\Cref{prop:error}}{Proposition 2.1}}\label{sec:error}
This section is dedicated to the proof of \cref{prop:error}. We fix a compact interval $\cI\subset\R$ throughout, and recall from \cref{subsec:notation} that $\|\cdot\|_\infty \coloneqq \|\cdot\|_{L^{\infty}(\cI)}$. 
Recall also $I_n \coloneqq (e^n, e^{n+1}]$.

In what follows, we make repeated use of the following consequence of \cref{lem:series-summability}.
\begin{corollary}\label{cor:series-summability-dos}
Let $(A_p)_p$ be as in Case \caseS or Case \caseR, and fix a collection of scalars $(c_p)_p \subset \C$ in Case \caseS or $(c_p)_p \subset \R$ in Case \caseR satisfying
\begin{align}\label{eq:block-summability-condition-dos}
    \sum_{n\in\N} \Big( \sum_{p\in I_n} |c_p|^2\Big)^{1/2} <\infty\,.
\end{align}
Fix $j\geq 1$.  Then the series 
\[
    F_{N,\ep}(x) \coloneqq \sum_{p \leq N} \frac{c_p}{p^{j\ep}}  A_p p^{-ijx}
\]
satisfies
\begin{align}\label{eq:summability-dos-expmoment}
    \E \exp \, (\lambda \sup_{N,\ep \geq 0 } \|F_{N,\ep}\|_\infty)<\infty
\end{align}
for all $\lambda \geq 0$. Additionally, $F_{N,\ep}$
converges uniformly over $\cI$ almost surely as $\min(N,1/\ep)\to\infty$. 
\end{corollary}
\begin{proof}
We start by proving the corollary for $F_{N,0}$. Consider the block sum
\[
    F_{P,n,\ep}(x) \coloneqq \sum_{p \in (e^n, P)} \frac{c_p}{p^{j\ep}}A_p  p^{-ijx} = e^{-ij nx} \sum_{p \in (e^n, P)} A_p f_{p,\ep}(x) \,, \ \  \quad f_{p,\ep}(x) \coloneqq \frac{c_p}{p^{j\ep}} e^{-ij (\log p - n)x}\,.
\]
For $p \in I_n$, we have $p \leq e^{n+1}$, and thus $\|f_{p,0}\|_{H^1}^2 \ll |c_p|^2 (1+\log p-n)^2 \ll |c_p|^2$. Applying \cref{lem:series-summability} with $f_p \coloneqq f_{p,0}$, we find the block summability condition \eqref{eq:block-summability-condition} is satisfied, and thus 
\[
    \E \exp\Big( \lambda \sum_{n \in \N} \max_{P \in I_n} \|F_{P,n,0}\|_\infty\Big) <\infty\,.
\]
Now, observe the same  partial summation trick leading to \eqref{eq:partial-summation-trick} implies 
\[
    \max_{P \in I_n} \|F_{P,n,\ep}\|_\infty \leq \max_{P \in I_n} \|F_{P,n,0}\|_\infty\,.
\]
The above two displays imply \eqref{eq:summability-dos-expmoment} as well as convergence of $F_{N,\varepsilon}$, concluding the proof.
\end{proof}

\begin{proof}[Proof of \cref{prop:error}]
Recall the definition of $D_N$ from \eqref{eq:TruncatedEP}.
Observe  the twist condition \eqref{eq:TwistCondition} implies that the quantities $|f(p)|/\sqrt{p}$, $|f(p^2)|/p$, and $\sum_{j\geq 3} |f(p^j)|/p^{j/2}$ all tend to $0$  as $p \to \infty$. With this in mind, we fix $N_0 \in \N$ large such that
\begin{align}\label{def:N0}
     \sup_{p > N_0} \sup_{\ep \geq 0} \Bigg\{\sum_{j\geq 1}  \Bigg| \frac{\alpha_p^j f(p^j)}{p^{j(\frac12+\ep+ix)}}\Bigg|  \ \vee \ \sum_{j\geq 1} \Bigg| \frac{\alpha_p^j f(p)^j}{p^{j(\frac12+\ep+ix)}}\Bigg| \Bigg\} < 1 \,,
\end{align}
where we write $s \vee t \coloneqq \max(s,t)$.
The above choice of $N_0$ makes $1+g_p(x)$ bounded away from $0$ for all $x\in \R$, for various functions $g_p: \R\to\C$, so that the Taylor expansion of $\log (1+g_p(x))$ converges. 
In particular, define 
\[
    \cC_{N,\varepsilon}(x)\coloneqq    \prod_{N_0 < p\leq N}\bigg(1+\sum_{j\geq1}\frac{\alpha_p^j f(p^j)}{p^{j(\frac12+\varepsilon+ix)}}\bigg)\bigg(1-\frac{\alpha_pf(p)}{p^{\frac12+\varepsilon+ix}}\bigg)\,, \qquad
    S_{N,\varepsilon}^{j\geq 2}(x) \coloneqq \sum_{N_0 < p\leq N} \sum_{j\geq 2}\frac{f(p)^j}{jp^{\frac{j}{2}+j\varepsilon}}\alpha_p^j \, p^{-ijx}\,.
\]
We then have by geometric series expansion and Taylor expansion of the logarithm:
\begin{align*}
    D_N(\alpha f, \tfrac12 + \ep + ix) 
    &=
    D_{N_0}(\alpha f, \tfrac12+\ep+ix) \
    \cC_{N,\ep}(x)  
    \prod_{N_0 < p \leq N} \bigg(1- \frac{\alpha_p f(p)}{p^{\frac12+\ep+ix}}\bigg)^{-1}
    \\
    &=D_{N_0}(\alpha f, \tfrac12+\ep+ix) \
    \cC_{N,\ep}(x) \
    e^{S_{N,\ep}(x) + 
    S_{N,\ep}^{\geq 2}(x)- S_{N_0,\ep}(x)} \,.
\end{align*}
Rearranging, we obtain a decomposition for $E_{N,\ep}^{(1)}(x) := |D_N(\alpha f, \frac12+\ep+ix)|^2/\exp(2\Re S_{N,\ep}(x))$:
\[
    E_{N,\ep}^{(1)}(x) = \big| D_{N_0} (\alpha f, \tfrac12+\ep+ix) \big|^2  \cdot 
    e^{-2\Re S_{N_0, \ep}(x)}   \cdot 
    |\cC_{N,\ep}(x)|^2  \cdot 
    e^{2\Re S_{N,\ep}^{\geq 2}(x)}\,.
\]
It is clear the first two terms on the right-hand side above converge uniformly as $\min(N,1/\ep)\to\infty$ to $|D_{N_0} (\alpha f, \tfrac12+ix)|^2 \exp(-2\Re S_{N_0, 0}(x))$, having no $N$ dependence and being continuous in $\ep$.
Since $\exp(-2\Re S_{N_0, 0}(x))$ is non-vanishing, the proposition will be proved once we show
\begin{align}
    \log | \cC_{N,\ep}(x)|^2 + 2\Re S_{N,\ep}^{\geq 2}(x) \qquad \text{converges almost surely, uniformly over $x\in \cI$}
\end{align}
(noting that $\log | \cC_{N,\ep}(x)|^2$ is a continuous function for all $N,\ep \geq 0$ by \eqref{def:N0}) and, for every $\lambda >0$,
    \begin{align}
        \mathbb{E} \big[\sup_{N,\varepsilon\ge 0} \|\mathcal{C}_{N,\varepsilon}\|_{\infty}^{2\lambda}\big]&<\infty\,, \label{eq:CNExpo}\\
        \mathbb{E}\big[\exp\big(\lambda \sup_{N, \varepsilon\ge 0} \|\Re S_{N,\varepsilon}^{j\ge 2}\|_{\infty}\big)\big]&<\infty\,.\label{eq:SnExp}
    \end{align}

We proceed by analyzing $\cC_{N,\ep}$.  We have
\begin{align*}
    \cC_{N,\varepsilon}(x) \ = \ \prod_{N_0 < p\leq N}\bigg(1+\sum_{j\geq2}\frac{f(p^j)\alpha_p^j-f(p^{j-1})f(p)\alpha_p^j}{p^{j(\frac12+\ep+ix)}}\bigg) \ = \ \prod_{N_0 < p\leq N}\big (1+a_{p,\varepsilon}(x)+b_{p,\varepsilon}(x)\big)\,,
\end{align*}
where $a_{p,\ep}$ and $b_{p,\ep}$ are the $j=2$ and $j\geq3$ contributions, respectively:
\[
    a_{p,\ep}(x) \coloneqq \frac{f(p^2)-f(p)^2}{p^{1+2\ep}}\, \alpha_p^2\, p^{-2ix} \qquad \text{and} \qquad
    b_{p,\ep}(x) \coloneqq \sum_{j\geq3}\frac{f(p^j)-f(p^{j-1})f(p)}{p^{\frac{j}{2}+j\varepsilon}}\, \alpha_p^j \, p^{-ijx}\,.
\]
Then, similar to  the calculations below \eqref{eq:DomConvPre}, we find
\begin{align}\label{eq:cC-to-a-and-b}
    \Big\| \log \big|\cC_{N,\varepsilon}\big|^2 - 2\Re\sum_{N_0 < p\leq N} a_{p,\ep} \Big\|_\infty \ll \sum_{N_0 < p\leq N} \sup_{\ep \geq 0} \Big(\|a_{p,\ep}\|_\infty^2 + \|b_{p,\ep}\|_\infty + \|b_{p,\ep}\|_\infty^2\Big)\,.
\end{align}
Thus, uniform convergence of $\log |\cC_{N,\ep}|^2$ follows if we show uniform  convergence of $\sum_{p} a_{p,\varepsilon}$ as $\min(N,1/\varepsilon)\to\infty$ as well as convergence of the right-hand side above as $N\to\infty$, almost surely.

We start by showing there exists a deterministic constant $C>0$ such that
\begin{equation}\label{eq:bpe-bound}
\sum_{p}\sup_{\varepsilon\geq0} \|b_{p,\varepsilon}(x)\|_\infty \leq C\,.
\end{equation}
To show \eqref{eq:bpe-bound}, we use the triangle inequality and $|\alpha_p| = 1$ to write 
\begin{align*}
    \sum_{p}\sup_{\varepsilon\geq0}\|b_{p,\varepsilon}\|_{\infty}&\leq\sum_{p}\sum_{j\geq3}\frac{|f(p^j)| + |f(p^{j-1})f(p)|}{p^{j/2}}\,.
\end{align*}
By~\eqref{eq:TwistCondition} we have $\sum_p\sum_{j\geq3}\vert f(p^j)\vert/p^{j/2} <\infty$. Thus,
\begin{equation}\label{eq:bpepSetUp}
     \sum_{p}\sup_{\varepsilon\geq0}\|b_{p,\varepsilon}\|_\infty\ll\sum_{p}\sum_{j\geq3}\frac{|f(p^{j-1})f(p)|}{p^{j/2}}= \sum_{p}\frac{\vert f(p^2)f(p)\vert}{p^{3/2}}+\sum_{p}\frac{|f(p)|}{\sqrt{p}}\sum_{j\geq3}\frac{\vert f(p^j)\vert}{p^{j/2}}\,.
\end{equation}
Since we have $|f(p)|/\sqrt{p}\ll 1$, our twist condition~\eqref{eq:TwistCondition} implies that
\begin{equation}\label{eq:bpepEasy}
    \sum_{p}\frac{|f(p)|}{\sqrt{p}}\sum_{j\geq3}\frac{\vert f(p^j)\vert}{p^{j/2}} \ll \sum_{p}\sum_{j\ge 3} \frac{|f(p^j)|}{p^{j/2}} <\infty\,.
\end{equation}
For the other term in \eqref{eq:bpepSetUp}, we divide into scales and then use Cauchy-Schwarz:
\begin{align*}
    \sum_{p}\frac{\vert f(p^2)f(p)\vert}{p^{3/2}}
    \ \leq \  \sum_{n\in \N} \bigg(\sum_{p\in I_n} \frac{|f(p^2)|^2}{p^2}\bigg)^{1/2} \ \bigg(\sum_{p\in I_n}\frac{|f(p)|^2}{p}\bigg)^{1/2}\,.
\end{align*}
The second term in the summand is bounded as  $\ll 1/\sqrt{n}$ by \eqref{eq:Vn-C/n-bound}. Finiteness of $\sum_{p} |f(p^2)f(p)| / p^{3/2}$ then follows from the  twist condition \eqref{eq:TwistCondition}. Substituting this and~\eqref{eq:bpepEasy} into \eqref{eq:bpepSetUp} yields \eqref{eq:bpe-bound}.

We now turn to analyzing  the sum over $a_{p,\ep}$ via \cref{cor:series-summability-dos}, taking 
$
    c_p \coloneqq (f(p^2)-f(p)^2)/p
$ and $j\coloneqq2$ 
and noting  $\alpha_p^2$ is still a Steinhaus random variable.
Towards verifying \eqref{eq:block-summability-condition-dos}, we compute 
\begin{align*}
    \sum_{n\in \N} \bigg(\sum_{p\in I_n} \Big| \frac{f(p^2)- f(p)^2}{p} \Big|^2\bigg)^{1/2}
    \ll \sum_{n\in \N} \bigg(\sum_{p\in I_n} \frac{|f(p^2)|^2}{p^2}\bigg)^{1/2}+\sum_{n\in \N} \bigg(\sum_{p\in I_n}\frac{|f(p)|^4}{p^2}\bigg)^{1/2}\,.
\end{align*}
The first term is finite by \eqref{eq:TwistCondition}. For the second term, \eqref{def:N0} implies $|f(p)|^4/p^2 \ll |f(p)|^3/p^{3/2}$, so 
\begin{align}\label{eq:thisguyisusefullater}
    \sum_{n\in\N} \bigg(\sum_{p\in I_n}\frac{|f(p)|^4}{p^2}\bigg)^{1/2} \ll \sum_{n\in\N}\bigg(\sum_{p\in I_n}\frac{|f(p)|^3}{p^{3/2}}\bigg)^{1/2} \ll \sum_{n\in\N} \bigg(\sum_{p\in I_n}\frac{|f(p)|^3}{p^{3/2}}\bigg)^{1/3}<\infty\,,
\end{align}
where the conclusion again follows from \eqref{eq:TwistCondition}. Therefore, \eqref{eq:block-summability-condition-dos} is satisfied, and so \cref{cor:series-summability-dos} yields uniform almost-sure convergence of $\sum_p a_{p,\ep}$. Also, \eqref{eq:block-summability-condition-dos} yields 
$\sum_p \sup_{\ep \geq 0} \|a_{p,\ep}\|_\infty^2 <\infty$. Together with \eqref{eq:cC-to-a-and-b} and \eqref{eq:bpe-bound}, we have uniform almost-sure convergence of $\log |\cC_{N,\ep}|^2$. \cref{eq:CNExpo} follows from \eqref{eq:cC-to-a-and-b} and \eqref{eq:summability-dos-expmoment} for $\sum_p a_{p,\ep}$.

It remains to show the same results for $\Re S_{N,\ep}^{j\geq 2}$. We write 
\[
    S_{N,\ep}^{j\geq 2} = S_{N,\ep}^{j=2} + S_{N,\ep}^{j\geq 3}\,, \quad S_{N,\ep}^{j=2} (x) \coloneqq \sum_{N_0< p\leq N}  \frac{f(p)^2}{2 p^{1+2\ep}} \alpha_p^2 p^{-2ix}\,, \quad S_{N,\ep}^{j\geq 3} (x) \coloneqq \sum_{N_0 <p\leq N} \sum_{j\geq 3} \frac{f(p)^j}{j p^{\frac{j}{2}+j\ep}} \alpha_p^j p^{-ijx}\,.
\]
For the $j\geq 3$ term, we have the deterministic bound using \eqref{def:N0}
\begin{align*}
\sum_{p}\sup_{\varepsilon\ge 0}\bigg\| \sum_{j\ge 3} \frac{\alpha_p^jf(p)^j}{jp^{\frac{j}{2}+j\varepsilon}}p^{-ijx}\bigg\|_\infty &\le \frac{1}{3}\sum_{p} \sum_{j\ge 3} \frac{|f(p)|^j}{p^{j/2}}<\infty\,
\end{align*}
which gives uniform convergence of $S_{N,\varepsilon}^{j\ge 3}$ and finite exponential moments.
For $S_{N,\ep}^{j=2}$, we apply \cref{cor:series-summability-dos} with $c_p = f(p)^2/2p$ and $j=2$. Indeed, 
\eqref{eq:block-summability-condition-dos} is satisfied due to \eqref{eq:thisguyisusefullater}.
\end{proof}

\appendix

\section{Universality of GMC: proof of \texorpdfstring{\cref{prop:gmc}}{Proposition 2.4}}
\label{sec:app}
Here we prove \cref{prop:gmc}. 
In the language of GMC theory, \cref{prop:gmc} says that the GMC measure $\mu_{\infty}^{\text{GMC}}$ associated to the log-correlated Gaussian field $2\Re G_{\infty}$, where
\[
    G_{\infty}(x) \coloneq \sum_p \frac{f(p)}{\sqrt{p}} Z_p \, p^{-ix}
\]
is universal with respect to the approximations $(2 \Re G_{N,\ep})_{N, \ep \geq 0}$, defined in \eqref{eq:Gau},  as $\min(N,1/\ep) \to \infty$.
Below, we fix a compact interval $\cI \subset \R$ and write $\cM(\cI)$ to denote the space of Radon measures on $\cI$ equipped with the topology of weak convergence. 

For $\theta \in (0,1]$, convergence in probability of $(\mu_{N,\ep}^{\text{GMC}})_{N,\ep}$ will follow from verifying the hypotheses of Theorem~4.4 in \cite{junnila-saksman}, which we do presently.\footnote{\cite[Proposition~1]{atherfold2025fourier} gives the analogous result for the Gaussian Fourier series, and its proof is similar.}
First, we construct the limiting measure $\mu_{\infty}^{\text{GMC}}$ as the in-probability limit of $\mu_{N,0}^{\text{GMC}}$ as $N\to\infty$.

\begin{lemma}\label{lem:gmc-critical-line-convergence}
For $\theta \in (0,1]$, the family of measures $(\mu^{\text{GMC}}_{N,0})_{N \geq 0}$ converges in probability in $\cM(\cI)$ as $N\to\infty$  to a non-trivial, non-atomic random measure denoted by $\mu_{\infty}^{\text{GMC}}$.
\begin{proof}
For  $\theta=1$,  this was shown in  \cite[Section 2.3.2]{goro-wong-3}, and their proof works for all $|f|^2\in \mathbf{P}_\theta$. 
For $\theta \in (0,1)$ and $|f|^2\in \mathbf{P}_\theta$, the  computations in \cite[Section 2.3.1]{goro-wong-3} still apply, and verify the hypotheses of \cite[Theorem~6.1]{SW20}, which in fact yields almost-sure convergence of $(\mu^{\text{GMC}}_{N,0})_{N\geq 0}$.
\end{proof}
\end{lemma}

Second, we show universality of $\mu_{\infty}^{\text{GMC}}$ as a limit \emph{in distribution} (not yet in probability).
\begin{lemma}\label{lem:gmc-universality-in-distribution}
    For $\theta \in (0,1]$, the family of measures $(\mu^{\text{GMC}}_{N,\varepsilon})_{N,\ep \geq 0}$
    converges in distribution in $\cM(\mathcal{I})$ to $\mu_{\infty}^{\text{GMC}}$ as $\min(N,1/\ep) \to \infty$.
\begin{proof}
Let $C_{N,\ep}(x,y)$ denote the covariance function of $2\Re G_{N,\ep}$.
According to \cite[Theorem~1.1]{junnila-saksman}, and given \cref{lem:gmc-critical-line-convergence}, it suffices to show the following: there exists a sequence $R=R(\varepsilon, N)\in\N$, such that $R(\varepsilon,N)\to\infty$ as $\min(N,1/\ep)\to\infty$, and $C_{N,\ep}$ satisfies
\begin{equation}\label{eq:criteria1}
        \sup_{N,\ep}\sup_{x,y}|C_{N,\ep}(x,y)-C_{R,0}(x,y)|\leq C<\infty 
\end{equation}
as well as
    \begin{equation}\label{eq:criteria2}
        \lim_{\min(N,1/\ep)\to\infty}\sup_{|x-y|\geq\delta}|C_{N,\ep}(x,y)-C_{R,0}(x,y)|=0 \,,\quad\text{for all $\delta>0$.}
    \end{equation}
We show \eqref{eq:criteria1} and \eqref{eq:criteria2} for $R \coloneqq \min(\lfloor \exp(1/\ep) \rfloor , N)$. 
The computation is similar to the proofs of \cite[Lemma~2.6]{goro-wong-2} and \cite[Lemma~1]{atherfold2025fourier}. All convergence statements are as $\min(N,1/\varepsilon)\to\infty$. 

Using Stieltjes integration by parts, we have 
\begin{align}
    \big|C_{N,\varepsilon}&(x,y)-C_{R,0}(x,y)\big|=2\Big|\sum_{p}\frac{|f(p)|^2}{p}\big(p^{-2\varepsilon}\mathbbm{1}_{p\leq N}-\mathbbm{1}_{p\leq R}\big)\cos(|x-y|\log p)\Big|\notag\\
    &=2\Big|\int_{2}^{N} p^{-1}\big(p^{-2\varepsilon }\mathbbm{1}_{p\leq N}-\mathbbm{1}_{p\leq R}\big)\cos(|x-y|\log p)\d(\theta\li(p)+\cE_{|f|^2}(p))\Big|\notag\\
    &\ll\Big|\int_{2}^{N}\theta\big(p^{-2\varepsilon }\mathbbm{1}_{p\leq N}-\mathbbm{1}_{p\leq R}\big)\cos(|x-y|\log p)\tfrac{\d p}{p\log p}\Big| \label{eq:line2}\\
    &\qquad+\Big|\,\Big[p^{-1}(1-p^{-2\varepsilon})\cE_{|f|^2}(p)\cos(|x-y|\log p)\Big]^{R}_{2}\,\Big|\label{eq:line3}\\
    &\qquad+\int_{2}^{R}\frac{\cE_{|f|^2}(p)}{p^2}\Big(\big|(1+2\varepsilon)p^{-2\varepsilon}-1\big|+(1-p^{-2\varepsilon})|x-y|\Big)\d p\label{eq:line4}\\
    &\qquad + \Big|\Big[p^{-1} p^{-2\varepsilon}\cE_{|f|^2}(p)\cos(|x-y|\log p) \Big]_{R}^N \Big|+\int_R^N \frac{\mathcal{E}_{|f|^2}(p)}{p^2}\big(p^{-2\varepsilon}(1+2\varepsilon)+p^{-2\varepsilon}|x-y|\big)\;\mathrm{d}p\,. \label{eq:line5}
\end{align}
We start by bounding \eqref{eq:line3}. Using that $\mathcal{E}_{|f|^2}(x) = o(x/\log x)$ and $|\cos(|x-y|\log p)|\le 1$ we get
\begin{equation*}
    \Big|\,\Big[p^{-1}(1-p^{-2\varepsilon})\cE_{|f|^2}(p)\Big]^{R}_{2} \,\Big|\ll o\Big(\frac{1}{\log R}\Big)+C\cdot (1-2^{-2\varepsilon}) \to 0\,.
\end{equation*}
For \eqref{eq:line4}, let $F(\ep;p)$ be the integrand in \eqref{eq:line4}. 
Using 
$
\max\{1-p^{-2\varepsilon},|(1-2\varepsilon)p^{-2\varepsilon}-1|\}\ll 2\varepsilon \log p
$,
\begin{align}\label{eq:FeppSplit}
    \int_{2}^{R}F(\varepsilon;p)\d p\ll\int_{2}^{\log R}\frac{|\cE_{|f|^2}(p)|}{p^2}(2\varepsilon\log p)(1+|x-y|)~\d p+\int_{\log R}^{\infty}\frac{|\cE_{|f|^2}(p)|}{p^2}(1+|x-y|)~\d p.
\end{align}
Due to condition \eqref{eq:remainder}, the second integral converges to 0 uniformly in  $x,y\in\mathcal{I}$. For the first integral, using $\cE_{|f|^2}(p)=o(p/\log p)$ and $\varepsilon\ll 1/\log R$, we see that it is uniformly bounded by $O( \log\log R/\log R)$, which converges to zero. Thus, \eqref{eq:line3} and \eqref{eq:line4} satisfy  \eqref{eq:criteria1} and \eqref{eq:criteria2}. 

For \eqref{eq:line5}, the first term is uniformly bounded by $O(R^{-1}|\cE_{|f|^2}(R)|+N^{-1}|\cE_{|f|^2}(N)|)\ll o(1/\log R)$, which converges to zero. Up to a universal constant, the second term is smaller than the second summand on the right hand side of \eqref{eq:FeppSplit} and thus also converges to zero.

Lastly, for \eqref{eq:line2}, we first compute the integral over $p\in[2,R)$ by changing variables $s=\log p$:
\begin{align}\label{eq:boundA3}
    \Big|\int_{2}^{R}\theta&\big(p^{-2\ep}\mathbbm{1}_{p\leq N}-\mathbbm{1}_{p\leq R}\big)\cos(|x-y|\log p)\tfrac{\d p}{p\log p}\Big|\leq\Big| \int_{\log2}^{\log R}\frac{e^{-2\ep s}-1}{s}\cos(|x-y|s)\d s\Big|\,.
\end{align}
For fixed $|x-y|>0$,  we use Taylor's theorem $e^{-2\ep s}-1=-2\ep s+O(\ep^2 s^2)$ and integration by parts to deduce that the right-hand side above is of size 
\begin{equation*}
    O(\ep)\cdot\Big|\int_{\log2}^{\log R}\cos(|x-y|s)\d s\Big|+O(\ep^2\log R)+O(\ep^2)\cdot\Big|\int_{\log2}^{\log R}\sin(|x-y|s)\d s\Big|=O(1/\log R),
\end{equation*}
thereby giving \eqref{eq:criteria2}. Moreover, it follows from the elementary bound $|e^{-2\varepsilon s}-1|\leq 2\varepsilon s$ that the right-hand side of \eqref{eq:boundA3} is  of size at most $4\varepsilon\log R=O(1)$. Thus, the integral up to $p=R$ satisfies \eqref{eq:criteria1} and \eqref{eq:criteria2}. We then consider the integral over $p\in[R,N]$ and notice that 
\begin{align*}
    \int_{R}^{N}\theta&\big(p^{-2\ep}\mathbbm{1}_{p\leq N}-\mathbbm{1}_{p\leq R}\big)\cos(|x-y|\log p)\tfrac{\d p}{p\log p}=\theta\int_{R}^{N}p^{-2\varepsilon}\cos(|x-y|\log p)\tfrac{\d p}{p\log p}.
\end{align*}
In fact, uniformly in $x,y\in\cI$, it is of size $O(\int_{R}^{N}p^{-1-2\ep}\frac{\d p}{\log p})\ll\frac{\ep}{\log R}(R^{-2\ep}+N^{-2\ep})\ll1/\log^2 R$, which converges to zero.
This verifies \eqref{eq:criteria1} and \eqref{eq:criteria2}.
\end{proof}
\end{lemma}

The final input  we need to apply \cite[Theorem~4.4]{junnila-saksman} is  that the fields $(2\Re G_{N,\ep})_{N,\ep\geq 0}$  can be obtained from $(2\Re G_{R,0})_{R\geq 0}$ via a \emph{linear regularization process} (LRP), defined in \cite[Definition~4.3]{junnila-saksman}\footnote{Definition~4.3 and Theorem~4.4 in \cite{junnila-saksman} are stated for a linear regularization indexed by one parameter $N$, whereas we have two parameters, $N$ and $\ep$. Their definition and result  address our setting by considering subsequences $(N_k,\ep_k)_{k\in\N}$ such that $\min(N_k, 1/\ep_k)\to \infty$ as $k\to\infty$.} and recalled below using their notation.

Let $C^{\gamma}(\mathcal{I})$ denote the set of $\gamma$-H\"older continuous functions. A family of operators $\cR_{N,\ep}:\cup_{0<\gamma<1}C^{\gamma}(\mathcal{I})\to C(\mathcal{I})$ is called an LRP for an approximating sequence $(X_R)_{R\in\N} \subset \cup_{0<\gamma<1}C^{\gamma}(\mathcal{I})$
if
\begin{enumerate}
    \item For any $h\in\cup_{0<\gamma<1}C^{\gamma}(\mathcal{I})$, we have $\lim_{\min(N,1/\ep)\to\infty}\|\cR_{N,\ep} h-h\|_\infty=0$.
    \item The limit $\cR_{N,\ep} X \coloneqq \lim_{R\to\infty} \cR_{N,\ep} X_R$ exists in $C(\cI)$ almost surely.
\end{enumerate}

\begin{proof}[Proof of Proposition~\ref{prop:gmc}]
We first show convergence in probability of $\mu_{N,\ep}^{\text{GMC}}$ to the measure $\mu_{\infty}^{\text{GMC}}$ constructed in \cref{lem:gmc-critical-line-convergence} for all $\theta \in (0,1]$.
Given \cref{lem:gmc-critical-line-convergence,lem:gmc-universality-in-distribution}, this follows from \cite[Theorem~4.4]{junnila-saksman} once we show  $2\Re G_{N,\ep}$ is obtained from $2\Re G_{N,0}$ via an LRP.

Towards this, we take our approximating sequence $(X_R)_{R\in \N}$ to be
\begin{align*}
    X_R(x) &\coloneqq 2 \Re G_{R,0}(x)  =  \sum_{p\leq R} \bigg(\frac{f(p)}{\sqrt{p}}Z_p e^{-ix \log p} + \overline{\frac{f(p)}{\sqrt{p}}Z_p } e^{ix \log p}  \bigg)\,.
\end{align*}
We define the operators $\cR_{N,\ep}$ by their action on trigonometric polynomials: for all $\xi \in \R$, 
\begin{align*}
   \cR_{N,\ep} \cos(\xi x) = \ind{|\xi| \leq \log N} e^{-\ep |\xi|} \cos(\xi x) \,, \qquad \cR_{N,\ep} \sin(\xi x) = \ind{|\xi| \leq \log N} e^{-\ep |\xi|} \sin(\xi x)\,.
\end{align*}
Defining $\cR_{N,\ep}$ on complex-valued functions $f$ via the complexification $\cR_{N,\ep} f = \cR_{N,\ep} [\Re f] + i \cR_{N,\ep} [\Im f]$, we see that $\cR_{N,\ep}$ is a Fourier multiplier:
$\cR_{N,\ep} e^{i\xi x} =  \ind{|\xi| \leq \log N} e^{-\ep |\xi|} e^{i\xi x}$. In particular, we have 
\begin{align}\label{eq:lin-reg-good}
    \cR_{N,\ep} X_R = 2 \Re G_{N \wedge R, \ep} \,,
\end{align}
where $a\wedge b \coloneqq \min(a,b)$.
From this, the operators $\cR_{N,\ep}$ can be trivially extended to all of $\mathscr{C} \coloneqq \cup_{0<\gamma <1} C^{\gamma}(\cI)$. To be precise,  since the space
\[
    \mathscr{E} \coloneqq \mathrm{span}_{\R}\{1,\cos(\xi x), \sin(\xi x) : \xi >0 \} \subset \mathscr{C}
\]
and $\mathscr{C}$ are both vector spaces and   $\cR_{N,\ep}$ has been defined on a Hamel basis of $\mathscr{E}$, one can then extend this Hamel basis to a Hamel basis of $\mathscr{C}$. Writing $\mathscr{C} = \mathscr{E} \oplus \widetilde{\mathscr{E}}$, we   define $\cR_{N,\ep} \tilde{f} = \tilde{f}$ for all $\tilde{f}\in \widetilde{\mathscr{E}}$. This completes the definition of $\cR_{N,\ep}$ as an operator from $\mathscr{C}$ to $C(\cI)$.

To verify that $(\cR_{N,\ep})_{N,\ep}$ is an LRP\footnote{Though not mentioned explicitly in \cite[Definition~4.3]{junnila-saksman}, the proof of \cite[Theorem~4.4]{junnila-saksman}  uses that $\cR_{N,\ep}$ are continuous operators from $C^{\gamma}(\cI)\to C(\cI)$ in its application of the Banach--Steinhaus theorem.  Our $\cR_{N,\ep}$ are not necessarily bounded on all of $C^{\gamma}(\cI)$.
However, the Banach--Steinhaus theorem is only used to show that, for each fixed $R>0$, $\sup_{N,\ep \geq 0} \|\cR_{N,\ep} X_R\|_{\infty} <\infty$ (see the second-to-last display in \cite[Page~13]{junnila-saksman}). For us, this is immediate from \eqref{eq:lin-reg-good}.},  fix $h=f+g\in\mathscr{C}$, where $f\in \cE$ and $g \in \widetilde{\cE}$. 
Suppose we can write, for some $k\in \N$ and  $a_i, b_i, \xi_i, c \in \R$,
\[
    f(x) =c+\sum_{i=1}^k(a_i \cos(\xi_i x)+b_i\sin(\xi_i x))\,.
\]
Direct calculation implies
\begin{equation*}
    \Vert\cR_{N,\ep} h-h\Vert_\infty = \Vert\cR_{N,\ep} f-f\Vert_\infty
    \leq \sum_{i=1}^k (|a_i|+|b_i|)\cdot|e^{-\ep|\xi_i|}-1| \ind{|\xi_i| \leq \log N} +\sum_{i=1}^k(|a_i|+|b_i|) \ind{|\xi_i| > \log N}
\end{equation*}
Since $k$ is fixed, the right-hand side converges to 0 as $\min(N,1/\ep)\to\infty$. This shows  (1) in the above definition of LRPs. For (2), we see from \eqref{eq:lin-reg-good} that $\cR_{N,\ep} X_R-2 \Re G_{N, \ep}=0$ for all $R\geq N$.

Kahane's convexity inequality and the $L^r$ bounds in~\cite[Sec.~5.3]{10.1214/09-AOP490} imply $L^r$ convergence. 
\end{proof}

\printbibliography

@article{BerestyckiSimplePath,
author = {Nathana{\"e}l Berestycki},
title = {{An elementary approach to Gaussian multiplicative chaos}},
volume = {22},
journal = {Electronic Communications in Probability},
%number = {none},
publisher = {Institute of Mathematical Statistics and Bernoulli Society},
pages = {1 -- 12},
keywords = {Gaussian free field, Gaussian multiplicative chaos, Liouville quantum gravity, Thick points},
year = {2017},
doi = {10.1214/17-ECP58},
%URL = {https://doi.org/10.1214/17-ECP58}
}

@misc{berestycki-powell,
      title={Gaussian free field and Liouville quantum gravity}, 
      author={Nathanaël Berestycki and Ellen Powell},
      year={2024},
      eprint={2404.16642},
      archivePrefix={arXiv},
      primaryClass={math.PR}
}

@article{RV14,
author = {R{\'e}mi Rhodes and Vincent Vargas},
title = {{Gaussian multiplicative chaos and applications: A review}},
volume = {11},
journal = {Probability Surveys},
%number = {none},
publisher = {Institute of Mathematical Statistics and Bernoulli Society},
pages = {315 -- 392},
keywords = {Gaussian multiplicative chaos, Gaussian process, KPZ, multifractal measures, review},
year = {2014},
doi = {10.1214/13-PS218},
%URL = {https://doi.org/10.1214/13-PS218}
}

@article{kahanegmc,
    author = {Kahane, Jean-Pierrre},
    title = {Sur le chaos multiplicatif},
    journal = {Ann. Sci. Math. Qu\'{e}bec},
    volume = {9},
    number = {2},
    pages = {105-150},
    year = {1985}
}

@article{PowellCrit,
  title={Critical Gaussian multiplicative chaos: a review},
  author={Powell, Ellen},
  journal={Markov Processes and Related Fields},
  volume={27},
  number={4},
  year={2021},
  publisher={Polymat}
}

@article{shamov2016gaussian,
  title={On Gaussian multiplicative chaos},
  author={Shamov, Alexander},
  journal={Journal of Functional Analysis},
  volume={270},
  number={9},
  pages={3224--3261},
  year={2016},
  publisher={Elsevier},
  doi = {10.1016/j.jfa.2016.03.001},
}

@book{KahaneBook,
  author    = {Jean-Pierre Kahane},
  title     = {Some Random Series of Functions},
  series    = {Cambridge Studies in Advanced Mathematics},
  volume    = {5},
  publisher = {Cambridge University Press},
  edition   = {Second},
  year      = {1985}
}

@book{LedouxTalagrand,
  author    = {Michel Ledoux and Michel Talagrand},
  title     = {Probability in {B}anach Spaces: Isoperimetry and Processes},
  series    = {Ergebnisse der Mathematik und ihrer Grenzgebiete},
  volume    = {23},
  publisher = {Springer},
  year      = {1991}
}

@book{kahane-khintchine-source,
  author    = {Tuomas Hyt{\"o}nen and Jan van Neerven and Mark Veraar and Lutz Weis},
  title     = {Analysis in Banach Spaces. Volume II: Probabilistic Methods and Operator Theory},
  series    = {Ergebnisse der Mathematik und ihrer Grenzgebiete. 3. Folge},
  volume    = {67},
  publisher = {Springer},
  year      = {2017},
  doi       = {10.1007/978-3-319-69808-3}
}

@article{vallender,
author = {Vallender, S. S.},
title = {Calculation of the Wasserstein Distance Between Probability Distributions on the Line},
journal = {Theory of Probability \& Its Applications},
volume = {18},
number = {4},
pages = {784-786},
year = {1974},
doi = {10.1137/1118101},

URL = { 
    
        https://doi.org/10.1137/1118101
    
    

},
eprint = { 
    
        https://doi.org/10.1137/1118101
    
    

}
}

@book{bhattacharya-rao,
author = {Bhattacharya, Rabi N. and Rao, R. Ranga},
title = {Normal Approximation and Asymptotic Expansions},
publisher = {Society for Industrial and Applied Mathematics},
year = {2010},
doi = {10.1137/1.9780898719895},
address = {},
edition   = {},
URL = {https://epubs.siam.org/doi/abs/10.1137/1.9780898719895},
eprint = {https://epubs.siam.org/doi/pdf/10.1137/1.9780898719895}
}

@article{rotar,
author = {Rotar’, V. I.},
title = {A Non-Uniform Estimate for the Convergence Speed in the Multi-Dimensional Central Theorem},
journal = {Theory of Probability \& Its Applications},
volume = {15},
number = {4},
pages = {630-648},
year = {1970},
doi = {10.1137/1115072},

URL = { 
    
        https://doi.org/10.1137/1115072
    
    

},
eprint = { 
    
        https://doi.org/10.1137/1115072
    
    

}
}

@article{SW20,
author = {Eero Saksman and Christian Webb},
title = {{The Riemann zeta function and Gaussian multiplicative chaos: Statistics on the critical line}},
volume = {48},
journal = {The Annals of Probability},
number = {6},
publisher = {Institute of Mathematical Statistics},
pages = {2680 -- 2754},
year = {2020},
doi = {10.1214/20-AOP1433}
}

@article{jego21,
  author = {Jego, Antoine},
  da = {2021/06/01},
  doi = {10.1007/s00440-021-01051-7},
  id = {Jego2021},
  journal = {Probability Theory and Related Fields},
  number = {1},
  pages = {495--552},
  title = {Critical Brownian multiplicative chaos},
  ty = {JOUR},
  volume = {180},
  year = {2021}
}

@misc{KK24,
      title={Absolute continuity of non-{G}aussian and {G}aussian multiplicative chaos measures}, 
      author={Yujin H. Kim and Xaver Kriechbaum},
      year={2024},
      eprint={2410.19979},
      archivePrefix={arXiv},
      primaryClass={math.PR}
}

@misc{lacoin22,
      title={{Critical Gaussian Multiplicative Chaos revisited}}, 
      author={Hubert Lacoin},
      year={2022},
      eprint={2209.06683},
      archivePrefix={arXiv},
      primaryClass={math.PR} 
}

@misc{goro-wong-1,
      title={Martingale central limit theorem for random multiplicative functions}, 
      author={Ofir Gorodetsky and Mo Dick Wong},
      year={2024},
      eprint={2405.20311},
      archivePrefix={arXiv},
      primaryClass={math.NT},
      url={https://arxiv.org/abs/2405.20311}
}

@misc{goro-wong-2,
      title={Multiplicative chaos measure for multiplicative functions: the $L^1$-regime}, 
      author={Ofir Gorodetsky and Mo Dick Wong},
      year={2025},
      eprint={2503.10555},
      archivePrefix={arXiv},
      primaryClass={math.NT}
}

@misc{goro-wong-3,
      title={On the limiting distribution of sums of random multiplicative functions}, 
      author={Ofir Gorodetsky and Mo Dick Wong},
      year={2025},
      eprint={2508.12956},
      archivePrefix={arXiv},
      primaryClass={math.NT}
}

@inproceedings{harper2020moments,
  title={Moments of random multiplicative functions, I: Low moments, better than squareroot cancellation, and critical multiplicative chaos},
  author={Harper, Adam J},
  booktitle={Forum of Mathematics, Pi},
  volume={8},
  pages={e1},
  year={2020},
  organization={Cambridge University Press}
}

@article{helson2010hankel,
  title={Hankel forms},
  author={Helson, Henry},
  journal={Studia Mathematica},
  volume={198},
  pages={79--84},
  year={2010},
  publisher={Instytut Matematyczny Polskiej Akademii Nauk}
}

@article{atherfold2025fourier,
  title={The Fourier coefficients of the critical holomorphic multiplicative chaos},
  author={Atherfold, Christopher and Najnudel, Joseph},
  journal={arXiv preprint arXiv:2508.13849},
  year={2025}
}

@article{junnila-saksman,
author = {Janne Junnila and Eero Saksman},
title = {{Uniqueness of critical Gaussian chaos}},
volume = {22},
journal = {Electronic Journal of Probability},
number = {none},
publisher = {Institute of Mathematical Statistics and Bernoulli Society},
pages = {1 -- 31},
year = {2017},
doi = {10.1214/17-EJP28},
URL = {https://doi.org/10.1214/17-EJP28}
}

@article{DRSV14a,
author = {Bertrand Duplantier and R{\'e}mi Rhodes and Scott Sheffield and Vincent Vargas},
title = {{Critical Gaussian multiplicative chaos: Convergence of the derivative martingale}},
volume = {42},
journal = {The Annals of Probability},
number = {5},
publisher = {Institute of Mathematical Statistics},
pages = {1769 -- 1808},
year = {2014},
doi = {10.1214/13-AOP890},
URL = {https://doi.org/10.1214/13-AOP890}
}

@article{DRSV14b,
	author = {Duplantier, Bertrand and Rhodes, R{\'e}mi and Sheffield, Scott and Vargas, Vincent},
	date = {2014},
	doi = {10.1007/s00220-014-2000-6},
	id = {Duplantier2014},
	journal = {Communications in Mathematical Physics},
	number = {1},
	pages = {283--330},
	title = {Renormalization of Critical Gaussian Multiplicative Chaos and KPZ Relation},
	url = {https://doi.org/10.1007/s00220-014-2000-6},
	volume = {330},
	year = {2014}}

@article{APS19,
author = {Juhan Aru and Ellen Powell and Avelio Sep{\'u}lveda},
title = {{Critical Liouville measure as a limit of subcritical measures}},
volume = {24},
journal = {Electronic Communications in Probability},
number = {none},
publisher = {Institute of Mathematical Statistics and Bernoulli Society},
pages = {1 -- 16},
year = {2019},
doi = {10.1214/19-ECP209},
URL = {https://doi.org/10.1214/19-ECP209}
}

@article{powell18,
author = {Ellen Powell},
title = {{Critical Gaussian chaos: convergence and uniqueness in the derivative normalisation}},
volume = {23},
journal = {Electronic Journal of Probability},
number = {none},
publisher = {Institute of Mathematical Statistics and Bernoulli Society},
pages = {1 -- 26},
year = {2018},
doi = {10.1214/18-EJP157},
URL = {https://doi.org/10.1214/18-EJP157}
}

@article{10.1214/09-AOP490,
author = {Raoul Robert and Vincent Vargas},
title = {{Gaussian multiplicative chaos revisited}},
volume = {38},
journal = {The Annals of Probability},
number = {2},
publisher = {Institute of Mathematical Statistics},
pages = {605 -- 631},
year = {2010},
doi = {10.1214/09-AOP490},
URL = {https://doi.org/10.1214/09-AOP490}
}

@article{wintner,
author = {Aurel Wintner},
title = {{Random factorizations and Riemann’s hypothesis}},
volume = {11},
journal = {Duke Mathematical Journal},
number = {2},
publisher = {Duke University Press},
pages = {267 -- 275},
year = {1944},
doi = {10.1215/S0012-7094-44-01122-1},
URL = {https://doi.org/10.1215/S0012-7094-44-01122-1}
}

@article{granville2001large,
  title={Large character sums},
  author={Granville, Andrew and Soundararajan, Kannan},
  journal={Journal of the American Mathematical Society},
  volume={14},
  number={2},
  pages={365--397},
  year={2001}
}

@article{harper2023typical,
  title={The typical size of character and zeta sums is $o(\sqrt{x})$},
  author={Harper, Adam J},
  journal={arXiv preprint arXiv:2301.04390},
  year={2023}
}

@article{gorodetsky2025short,
  title={A short proof of Helson's conjecture},
  author={Gorodetsky, Ofir and Wong, Mo Dick},
  journal={Bulletin of the London Mathematical Society},
  volume={57},
  number={4},
  pages={1065--1076},
  year={2025},
  publisher={Wiley Online Library}
}

@article{BohrJessenI,
  author  = {Bohr, Harald and Jessen, B{\"o}rge},
  title   = {{\"U}ber die Wertverteilung der {R}iemannschen Zetafunktion. Erste Mitteilung},
  journal = {Acta Mathematica},
  volume  = {54},
  pages   = {1--35},
  year    = {1930}
}

@article{BohrJessenII,
  author  = {Bohr, Harald and Jessen, B{\"o}rge},
  title   = {{\"U}ber die Wertverteilung der {R}iemannschen Zetafunktion. Zweite Mitteilung},
  journal = {Acta Mathematica},
  volume  = {58},
  pages   = {1--55},
  year    = {1932}
}

@phdthesis{bagchi81,
  author  = {Bagchi, B.},
  title   = {Statistical Behaviour and Universality Properties of the
             {R}iemann Zeta Function and Other Allied {D}irichlet Series},
  school  = {Indian Statistical Institute},
  address = {Kolkata, India},
  year    = {1981},
  url     = {https://library.isical.ac.in/jspui/handle/10263/4256},
  note    = {Available from the Indian Statistical Institute Digital Repository}
}

@inbook{harper-icm,
author = {Adam J. Harper},
title = {Better than Squareroot Cancellation in Number Theory},
booktitle = {Proceedings of the International Congress of Mathematicians 2026 - Volume 3: Invited Lectures (Sections 1–4)},
chapter = {},
year = 2026,
pages = {396-413},
publisher = {Society for Industrial and Applied Mathematics},
doi = {10.1137/25M1826366},
URL = {https://epubs.siam.org/doi/abs/10.1137/25M1826366},
eprint = {https://epubs.siam.org/doi/pdf/10.1137/25M1826366}
}

@book{montgomery-vaughan-1, place={Cambridge}, series={Cambridge Studies in Advanced Mathematics}, title={Multiplicative Number Theory I: Classical Theory}, publisher={Cambridge University Press}, author={Montgomery, Hugh L. and Vaughan, Robert C.}, year={2006}, collection={Cambridge Studies in Advanced Mathematics}}

@INCOLLECTION{sound-survey,
  title     = "The distribution of values of zeta and L-functions",
  booktitle = "International Congress of Mathematicians",
  author    = "Soundararajan, Kannan",
  publisher = "EMS Press",
  pages     = "1260--1310",
  month     =  dec,
  year      =  2023,
  language  = "en"
}

@book{rudin1987real,
  title={Real and Complex Analysis},
  author={Rudin, W.},
  isbn={9780071002769},
  lccn={86000007},
  series={Mathematics series},
  year={1987},
  publisher={McGraw-Hill}
}

@ARTICLE{fyodorov-keating,
  title     = "Freezing transitions and extreme values: random matrix theory,
               $\zeta(1/2 + it)$ and disordered landscapes",
  author    = "Fyodorov, Yan V and Keating, Jonathan P",
  journal   = "Philos. Trans. A Math. Phys. Eng. Sci.",
  publisher = "The Royal Society",
  volume    =  372,
  number    =  2007,
  pages     = "20120503",
  month     =  jan,
  year      =  2014,
  language  = "en"
}

@article{fhk,
  title = {Freezing Transition, Characteristic Polynomials of Random Matrices, and the Riemann Zeta Function},
  author = {Fyodorov, Yan V. and Hiary, Ghaith A. and Keating, Jonathan P.},
  journal = {Phys. Rev. Lett.},
  volume = {108},
  issue = {17},
  pages = {170601},
  numpages = {5},
  year = {2012},
  month = {4},
  publisher = {American Physical Society},
  doi = {10.1103/PhysRevLett.108.170601},
  url = {https://link.aps.org/doi/10.1103/PhysRevLett.108.170601}
}

@ARTICLE{paquette-zeitouni,
  title     = "The extremal landscape for the {C} $\beta $ {E} ensemble",
  author    = "Paquette, Elliot and Zeitouni, Ofer",
  journal   = "For. Math. Sigma",
  publisher = "Cambridge University Press (CUP)",
  volume    =  13,
  number    = "e1",
  year      =  2025,
  language  = "en"
}

@misc{abr1,
      title={{The Fyodorov-Hiary-Keating Conjecture. I}}, 
      author={Louis-Pierre Arguin and Paul Bourgade and Maksym Radziwiłł},
      year={2020},
      eprint={2007.00988},
      archivePrefix={arXiv},
      primaryClass={math.PR},
      url={https://arxiv.org/abs/2007.00988}, 
}

@misc{abr2,
      title={{The Fyodorov-Hiary-Keating Conjecture. II}}, 
      author={Louis-Pierre Arguin and Paul Bourgade and Maksym Radziwiłł},
      year={2023},
      eprint={2307.00982},
      archivePrefix={arXiv},
      primaryClass={math.NT},
      url={https://arxiv.org/abs/2307.00982}, 
}

@misc{kim-kriechbaum-future,
      title={Non-Gaussian multiplicative chaos from sums of indepedendent increments: criticality and universality}, 
      author={Yujin H. Kim and Xaver Kriehcbaum}, 
      note = {Preprint, forthcoming}
}

@misc{email,
      title={Private communication}, 
      author={Ofir Gorodetsky and Mo Dick Wong}
}

\end{document}